\documentclass{article}
\usepackage[utf8x]{inputenc}
\usepackage{amsmath, amsthm, amssymb}
\usepackage{color}
\usepackage[margin=1.5in]{geometry}
\usepackage{dsfont}
\usepackage{graphicx}
\usepackage{caption}

\newtheorem{theorem}{Theorem}[section]
\newtheorem{proposition}{Proposition}[section]
\newtheorem{lemma}{Lemma}[section]
\newtheorem*{remark}{Remark}
\newtheorem{corollary}{Corollary}[section]
\newtheorem{definition}{Definition}[section]
\newcommand{\R}{\mathbb R}
\newcommand{\N}{\mathbb N}
\newcommand{\1}{\mathds{1}}

\title{Eventual Regularization of Fractional Mean Curvature Flow}
\author{Stephen Cameron}
\date{}

\begin{document}
\maketitle

\begin{abstract}
We show that any open set that is a finite distance away from a Lipschitz subgraph will become a Lipschitz subgraph after flowing under fractional mean curvature flow for a finite, universal time.  Our proof is quantitative and inherently nonlocal, as the corresponding statement is false for classical mean curvature flow.  This is the first regularizing effect proven for weak solutions to nonlocal curvature flow.  
\end{abstract}

\section{Introduction}

For sufficiently regular set $E\subseteq \R^d$ and $s\in (0,1)$, we define the $s$-fractional perimeter of $E$
\begin{equation}
\begin{split}
P_s(E):= s(1-s)[\1_{E}]_{\dot{W}^{s,1}} &= s(1-s)\int\limits_{\R^d}\int\limits_{\R^d} \frac{|\1_{E}(X)-\1_E(Y)|}{|X-Y|^{d+s}}dXdY
\\&= 2s(1-s)\int\limits_{\R^d}\int\limits_{\R^d} \frac{\1_{E}(X)\1_{\mathcal{C}E}(Y)}{|X-Y|^{d+s}}dXdY,
\end{split}
\end{equation}
where $\1_E, \1_{\mathcal{C}E}$ are the characteristic functions of $E$ and its compliment $\mathcal{C}E$.  The $s$-fractional perimeter interpolates between our usual notions of perimeter and Lebesgue measure, with 
\begin{equation}
\lim\limits_{s\to 1}P_s(E)=C_d P(E), \qquad \lim\limits_{s\to 0}P_s(E) =C_d'\mathcal{L}^d(E),
\end{equation} 
for bounded regular sets (see \cite{LimitS1}, \cite{LimitS0}).  The $s$-fractional perimeter was first introduced in \cite{NonlocalMinimal} where the authors studied the regularity of minimizers, known as nonlocal minimal surfaces.  Minimizers satisfy the Euler-Lagrange equation
\begin{equation}\label{e:smeancurvature}
H_s(X,E) := -s(1-s)P.V.\int\limits_{\R^d}\frac{\1_{E}^{\pm}(X+Z)}{|Z|^{d+s}}dZ=0,  
\end{equation}
for all points $X\in \partial E$, where $\1_{E}^\pm = \1_E - \1_{\mathcal{C}E}$ is the signed characteristic function.  The quantity $H_s$ is called the $s$-fractional mean curvature, and it converges to the classical mean curvature as $s\to 1$ \cite{LimitS1}.  Thus fractional mean curvature can be thought of as a nonlocal, fractional order analogue of local mean curvature.  

We are interested in studying the regularizing effects of the flow 
\begin{equation}\label{e:meancurvfloweqnset}
\partial_tX(t) = -H_s(X,E_t)\nu(X), \quad X(t)\in \partial E_t.
\end{equation}  
Solutions to the flow \eqref{e:meancurvfloweqnset} are translation invariant, satisfy the comparison principle, and have natural rescaling $t\to \displaystyle\frac{1}{R}E_{R^{1+s}t}$ for any $R>0$.  See Propositions \ref{p:comparisonprinciple}, \ref{p:fracmeancurvbasics}, and Corollary \ref{c:fracmeancurvflowrescaling} for basic proofs.

\begin{figure}[!htb]
\minipage{0.32\textwidth}
\includegraphics[width=\linewidth]{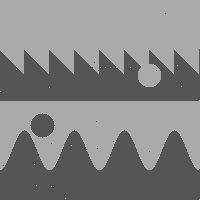}
\caption{\label{fig:initial} $t=0$}
\endminipage\hfill
\minipage{0.32\textwidth}
\includegraphics[width=\linewidth]{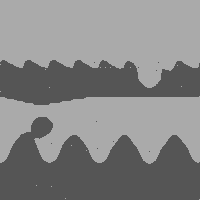}
\caption{$t=4$}
\endminipage\hfill
\minipage{0.32\textwidth}
\includegraphics[width=\linewidth]{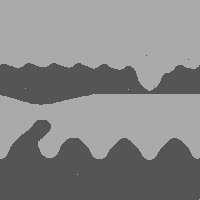}
\caption{\label{fig:pinch1} $t=7$}
\endminipage\hfill
\end{figure}

\begin{figure}[!htb]
\minipage{0.32\textwidth}
\includegraphics[width=\linewidth]{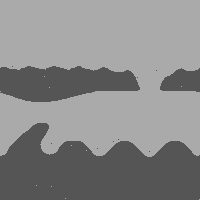}
\caption{\label{fig:pinch2} $t=10$}
\endminipage\hfill
\minipage{0.32\textwidth}
\includegraphics[width=\linewidth]{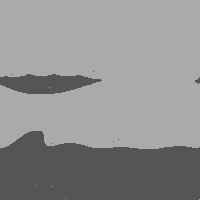}
\caption{$t=18$}
\endminipage\hfill
\minipage{0.32\textwidth}
\includegraphics[width=\linewidth]{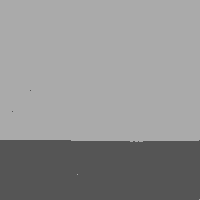}
\caption{\label{fig:final}$t=27$}
\endminipage\hfill
\caption*{$x$-periodic,  $s=0.2$ nonlocal mean curvature flow}
\end{figure}

Figures \ref{fig:initial} to \ref{fig:final} above give a good illustration of a number interesting properties of the nonlocal flow.  Between Figures \ref{fig:initial} and \ref{fig:pinch1}, we see the disjoint ball in the lower left hand corner become attracted to and join the lower portion of the set.  In the upper right hand corner of Figures \ref{fig:pinch1} and \ref{fig:pinch2}, we see a neckpinch singularity form in the upper portion of the set, which is impossible for 2d local mean curvature flow (see \cite{Neckpinch} for more details).  Finally after the neckpinch, we see the upper portion of the set shrink to nothing until eventually we are left with a flat graphical set in Figure \ref{fig:final}.

In the case that $E_0 $ is the subgraph of a smooth Lipschitz function $u_0: \R^{d-1}\to \R$, the flow exists for all time \cite{SmoothCurvatureEquations} and can be equivalently described by
\begin{equation}\label{e:meancurvfloweqngraph}
\partial_t u(t,x) = s(1-s)\sqrt{1+|\nabla u(t,x)|^2} \int\limits_{\R^{d-1}}\frac{u(t,x+z)-u(t,x)}{|z|^{d+s}}\Lambda\left(\frac{u(t,x+z)-u(t,x)}{|z|}\right)dz,
\end{equation}
where the nonlinearity $\Lambda(L) = \displaystyle\frac{1}{L}\int\limits_{-L}^{L} \frac{1}{(1+z_d^2)^{(d+s)/2}}dz_d$.  Thus 
\begin{equation}
\Lambda(||\nabla_x u||_{L^\infty})\leq \Lambda\left(\frac{u(t,x+z)-u(t,x)}{|z|}\right)\leq 2,
\end{equation}
so this is a nonlinear parabolic equation of order $1+s$, with the ellipticity constant depending on the Lipschitz constant of $u$.  
From this parabolicity, we see that fractional mean curvature flow is regularizing on Lipschitz subgraphs.  

While in general smooth solutions $t\to E_t$ of \eqref{e:meancurvfloweqnset} do not exist for all times $t$ for nongraphical initial data, it's possible to define weak viscosity solutions via the level set method which will exist for all time.  See \cite{CyrilFract,Threshold,Chambolle,CrystallineChambolle} or the appendix for details.  In this article, we shall show that for any initial data $E_0$ that is bounded between two Lipschitz subgraphs, the minimal viscosity supersolution will itself become a Lipschitz subgraph in finite time.  

\begin{theorem}\label{t:main}
Let $E_0\subseteq \R^d$ be an open set and $u_0\in \dot{W}^{1,\infty}(\R^{d-1})$ be a Lipschtiz function with $||\nabla_x u_0||_{L^\infty} = L$ and 
\begin{equation}\label{e:u0setinclusion}
\left\{(x,x_d)\bigg| x\in\R^{d-1}, x_d< u_0(x)-\frac{R}{2}\right\} \subseteq E_0\subseteq \left\{(x,x_d)\bigg| x\in\R^{d-1}, x_d < u_0(x)+\frac{R}{2}\right\},
\end{equation}
for some $R\geq 0$.  Let $E_t$ be the minimal viscosity supersolution of the flow \eqref{e:meancurvfloweqnset}, in the sense of Definition \ref{d:viscosityset}.  Then for all $t\geq R^{1+s}T(d,s,L)$, $E_t$ will be a $(1+L)$-Lipschitz subgraph.  The time $T$ can be bounded explicitly, with
\begin{equation}
T(d,s,L) \leq \frac{C(d)(1+L)^{d+s}}{s^2(1-s)},
\end{equation}
for some dimensional constant $C(d)$.  
\end{theorem}

We stress that there is no initial regularity assumption on $E_0$ in Theorem \ref{t:main}.  The boundary $\partial E_0$ can have positive measure, and the set $E_0$ does not even need to be connected, such as in Figure \ref{fig:initial} above.  Our only assumption is that $E_0$ is a finite distance in the Hausdorff metric from a Lipschitz subgraph, which amounts to assuming that our set $E_0$ only ``grows linearly at infinity."  This linear growth allows us to recover the regularizing effects of fractional mean curvature flow on Lipschitz functions \eqref{e:meancurvfloweqngraph} at large scales.  Effectively, we apply \eqref{e:meancurvfloweqngraph} on large scales $|z|>R$ where our equation will be uniformly parabolic, and then use this to propogate these growth bounds to smaller scales until we achieve Lipschitz regularity.

Taking $u_0\equiv 0$ and $R=2$, we get the immediate corollary 

\begin{corollary}\label{c:flat}
Let $E_0\subseteq \R^d$ be an open set with 
\begin{equation}\label{e:u0setinclusion}
\left\{(x,x_d)\bigg| x_d< -1\right\} \subseteq E_0\subseteq \left\{(x,x_d)\bigg| x_d < 1\right\}.
\end{equation}
Let $E_t$ be the minimal viscosity supersolution of the flow \eqref{e:meancurvfloweqnset}, in the sense of Definition \ref{d:viscosityset}.  Then for all $t\geq T(d,s)$, $E_t$ will be a $1$-Lipschitz subgraph. 
\end{corollary}

From Corollary \ref{c:flat} it's clear that Theorem \ref{t:main} can be viewed as a parabolic version of the ``flat implies smooth" result of \cite{NonlocalMinimal} for nonlocal minimal surfaces.  One key difference between these results is that the proof of flat implies smooth in \cite{NonlocalMinimal} is by compactness, with non explicit constants.  Our proof is constructive, giving an explicit modulus of continuity for the set $E_t$.  See Subsection \ref{s:proofdiscussion} for a more detailed discussion of our approach.

Additionally, we wish to stress that the result Theorem \ref{t:main} is inherently nonlocal in nature, with the time $T(d,s,L)\to \infty$ as $s\to 1$.  This is a necessity, as the theorem is false for classical mean curvature flow.  The set 
\begin{equation}
E_0 = \{(x,x_d)| x_d<-1 \text{ or } 0<x_d<1\}, 
\end{equation}
is fixed by local mean curvature flow and hence never becomes graphical.  In this case, it's clear that the barrier to regularity is multiplicity, as the problem is that $\partial E_0$ is a disjoint union of hyperplanes.  But because of the nonlocal nature of fractional mean curvature, the points on the disjoint hyperplaces $\{x_d = \pm 1\}$ can still sense each other, and are no longer fixed.  Direct calculation shows that flowing under fractional mean curvature flow, $E_t \to \{(x,x_d)| x_d<0\}$ in finite time $T\sim \displaystyle\frac{1}{s(1-s)}$ for any $0<s<1$.

%

\begin{remark}(Higher Regularity) We note that for times $t\geq R^{1+s}T$, our set $E_t$ will be a uniformly Lipschitz subgraph, and hence the evolution can equivalently be described by the uniformly parabolic equation \eqref{e:meancurvfloweqngraph}.  This allows the use of general regularity results for nonlocal parabolic equations such as in \cite{NonlocalHolderRegularity} and \cite{NonlocalSchauder}.  Applying these results to successive derivatives may lead to a $C^\infty$ estimate, but we do not pursue that line of inquiry in this paper.  
\end{remark}

\subsection{Background}

Nonlocal perimeters arises naturally in the context of phase transition problems with very long range interactions.  In \cite{GammaConvergence}, the authors consider the energies 
\begin{equation}
\mathcal{E}_s(U):= \epsilon^{s}||U||_{H^{s/2}(\Omega)}+\int\limits_{\Omega} W(U(X))dX,
\end{equation}
where $s\in (0,2)$ and $W$ is a standard double well potential.  They show that after appropriate rescaling, the functionals $\Gamma$-converge as $\epsilon\to 0$ to classical perimeter functional for $s\in [1,2)$, but converge to the $s$-fractional perimeter in the more nonlocal case $s\in (0,1)$.  

The $s$-fractional mean curvature flow was first defined in \cite{Threshold}, where the authors were investigating the convergence of the threshold dynamics for the fractional heat equation $U_t+(-\Delta_X)^{s/2}U=0$.  In the case that $s\in [1,2)$, the evolution of the interface $\{U(t,\cdot)=0\}$ converged to classical mean curvature flow, but for $s\in (0,1)$ it instead converges to $s$-fractional mean curvature flow.  

Motivated by these applications and the parallels to classical minimal surfaces, there has been a sustained effort over the past 10 years to study the regularity of local minimizers of fractional perimeter, $s$-minimal surfaces.   \cite{NonlocalMinimal} began the study, recovering a number of the tools from classical minimal surfaces such as density estimates, monotonicity formula, and the improvement of flatness argument.  Nonlocal minimal surfaces are known to be smooth whenever they are Lipschitz \cite{NonlocalLip}, smooth outside of a set of codimension 2 for any $s\in (0,1)$ \cite{NonlocalMinimal}, and for $s$ sufficiently close to 1 the singularity set in fact has codimension at least 8 \cite{NonlocalLimit}, matching the regularity theory for the local case.  

There are however key differences between the regularity theory for the nonlocal and the local cases.  Stable nonlocal minimal surfaces satisfy a universal BV estimate \cite{QuantitativeFlatness}, which is false without additional assumptions for classical minimal surfaces, and an important open problem with the appropriate additional assumptions for dimension $d>3$.  There is also an example of a nontrivial stable $s$-minimal cone in $\R^7$ for small $s$ \cite{NonlocalCone}, showing the regularity of nonlocal minimal surfaces is different than the classical case for $s$ bounded away from 1.  It is still an open problem though if this is the case for minimizing nonlocal minimal surfaces.

Since nonlocal mean curvature flow's introduction in \cite{Threshold}, properties of smooth solutions have been studied in \cite{SmoothCurvatureEquations} and radial self-shrinkers in \cite{SelfShrinkers}.  Most work however on fractional mean curvature flow has focused on the study of weak solutions via the level set method and the singularities they develop.  

The level set method was popularized in \cite{LevelSetMethod}, where it was used as a numerical tool to study the evolution of classical mean curvature flow past the point of singularities. This was made analytically rigorous by \cite{EvansMCF1}, and is an invaluable tool in the study of mean curvature flow.  The key insight of the level set method is to replace the evolution of the boundary $t\to \partial E_t$ with the evolutions of the zero level set of a function $t\to \{U(t,\cdot)=0\}$, where $U$ now solves a degenerate parabolic equation based on the original flow.  

The existence, uniqueness, and comparison principal for global viscosity solutions defined via the level set method for fractional mean curvature flow was first shown by \cite{CyrilFract}, and then later expanded to more general nonlocal and even crystalline flows in \cite{Chambolle, CrystallineChambolle}.  We review the definitions and essential results in the appendix.  

One type of singularity particular to level set flows is ``fattening."  It refers to when the level set $\{U(t,\cdot)=0\}$ which represent our ``boundary" develops a nonempty interior.  This corresponds to a lack of uniqueness in the geometric flow, as the set $\{U(t,\cdot)=0\}$ is no longer a boundary of any set, and the two boundaries $\partial \{U(t,\cdot<0\} \not = \partial \{U(t,\cdot)>0\}$ are both equally valid evolutions of the original $\partial E_0$.  As at most countably many level sets $\{U(t,\cdot)=\gamma\}$ can fatten, it is in some sense a rare phenomena.  However there has a been an intense study to see what kind of properties of the initial set $E_0$ rule out the possibility.  For fractional mean curvature flow, it was shown in \cite{Neckpinch} that a smooth simple closed curve can fatten, in contrast to Grayson's theorem for classical mean curvature flow.  The preprint \cite{Fattening} goes through a number of illustrative examples of when fattening does or does not occur for nonlocal flows, proving smooth strictly star convex sets don't fatten.  The ``strictness" on the strictly star convex assumption is necessary though, as \cite{Fattening} also gives an example of a star convex set which does fatten.  

In this paper, we circumvent the issue of fattening by instead showing that for generic $\gamma\in \R$, the level set $\{U(t,\cdot)=\gamma\}$ becomes a Lipschitz graph in finite time.  By approximation, this allows us to show that the extremal solutions $\partial \{U(t,\cdot)<0\}$ and $\partial \{U(t,\cdot)>0\}$ each independently regularize.

\subsection{Argument Outline}\label{s:proofdiscussion}

Our proof of Theorem \ref{t:main} is inspired by Kiselev's proof of eventual H\"older regularization for solutions to the supercritical Burger's equation in \cite{KiselevSurvey}.  There, Kiselev shows that solutions to 
\begin{equation}
\begin{split}
\partial_t u(t,x) + u(t,x)\partial_x u(t,x) + (-\Delta)^\alpha u(t,x)=0, 
\\ \alpha\in (0,1/2).
\end{split}
\end{equation}
becomes H\"older continuous in finite time.  A priori, this is surprising as solutions to the equation are known to develop shocks \cite{KiselevBurgers}.  For the proof, he showed that the equation propagated a family of moduli of continuity $$\omega(t,r) \approx  \delta(t)+Cr^\beta ,$$ where $\delta(0) > 2||u_0||_{L^\infty}$ and $\delta(T_{\alpha,\beta}) = 0$.  Thus the moduli of continuity gives no new information at time $t=0$, controls the size of shocks for $0<t<T_{\alpha,\beta}$, and then forces the solution to become $\beta$-H\"older continuous at time $t=T_{\alpha,\beta}$.  

Allowing $\omega(t,0)>0$ lets us apply the concept of a modulus of continuity to a discontinuous function.  But it makes just as much sense to apply it in this case to a multivalued function like the boundary of a set which can fold over itself.  Then at time $t=T$, satisfying the modulus forces the boundary to be graphical.  Our goal is to construct an explicit time dependent family of moduli of continuity and show that exactly this occurs for fractional mean curvature flow.  See the beginning of the next sectionfor a useful illustration, and a precise definition of what it should mean for a set to satisfy a modulus of continuity.

In order to make our proof the most clear and understandable, we first prove Theorem \ref{t:main} in the case that the flow $t\to E_t$ is smooth.  We begin by defining what it means for a set to have a modulus of continuity in Section 2, and showing that our assumption \eqref{e:u0setinclusion} is equivalent to assuming our initial set has a Lipschitz modulus of continuity.  In Section 3, we repeat the breakthrough argument of \cite{KiselevSurvey} to set up an eventual proof by contradiction.  In section 4, we make a number of curvature estimates and reduce the proof by contradiction to the construction of a modulus of continuity satisfying an integral inequality.  In section 5, we construct that modulus of continuity and finish the proof by contradiction, completing the proof in the smooth case.  

In sections 6 and 7, we extend the smooth proof to work in the viscosity solution framework.  In section 6, we prove a number of technical lemmas in order to formally justify the break through argument of section 3 and estimates of section 4 for almost every level set $\{U(t,\cdot)=\gamma\}$ under the additional assumption that our initial set $E_0$ is asymptotically flat.  In section 7, we then apply limiting arguments to apply the result to the boundary of every level set without the flatness assumption, completing the general proof.


\section{Moduli of Continuity for Sets}

\begin{figure}[!htb]
\includegraphics[]{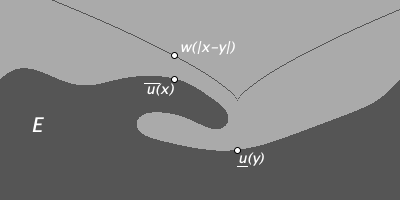}
\caption{\label{fig:modulus} A set $E$ with modulus $\omega(|x-y|)$}
\end{figure}

Our first step is to extend the idea of a modulus of continuity to a nongraphical set $E$.  With that in mind, we define what we call the upper and lower boundaries of a set by 
\begin{definition}\label{d:upperlowerboundary}(Upper and lower boundaries) Let $E\subseteq \R^d$.  Assume that for any $x\in \R^{d-1}$, the sets
\begin{equation}
\{x_d| (x,x_d)\in  \overline{E}\}, \qquad \{x_d| (x,x_d)\in \overline{\mathcal{C}E}\},
\end{equation}
are nonempty and bounded from above/below respectively.  Then we define the upper and lower boundaries of $E$ in the $x_d$-direction $\overline{u}, \underline{u}:\R^{d-1}\to \R$ to be 
\begin{equation}
\overline{u}(x) := \sup\{x_d| (x,x_d)\in  \overline{E}\}, \qquad \underline{u}(x):=\inf\{x_d| (x,x_d)\in \overline{\mathcal{C}E}\}.
\end{equation}
\end{definition}
Upper and lower boundaries could analogously be defined for any direction $e\in S^{d-1}$.  Without loss of generality, we restrict ourselves to the positive $x_d$-direction, which corresponds to thinking of our set $E$ as close to a subgraph.
  
Note that equivalently
\begin{equation}
\overline{u}(x) = \max\{x_d| (x,x_d)\in  \partial E\}, \qquad \underline{u}(x)=\min\{x_d| (x,x_d)\in \partial E\},
\end{equation}
once we know that our set $E$ both contains and is contained by a subgraph.  

\begin{definition} \label{d:boundarymodulusdefn}
Let $E\subseteq\R^d$ be a set with upper and lower boundaries in the $x_d$-direction, and $\omega:[0,\infty)\to [0,\infty)$ be a continuous function.  Then we say that $E$ has modulus $\omega$ in the $x_d$-direction if for all $x,y\in \R^{d-1}$, 
\begin{equation}
\overline{u}(x)-\underline{u}(y) \leq \omega(|x-y|).
\end{equation}
\end{definition}
See Figure \ref{fig:modulus} for a helpful visual.

Note that we don't force $\omega(0)=0$ in our definition of a modulus of continuity, which allows for this definition to make sense when $\partial E$ is not graphical.  Indeed, if $E$ has modulus $\omega$ with $\omega(0)=0$, then necessarily we have that 
\begin{equation}
\overline{u}(x) = \underline{u}(x) \ \forall x\in\R^{d-1}\quad \Rightarrow \quad \partial E = \text{graph}(u:\R^{d-1}\to \R).
\end{equation}

To begin, we first note that $\overline{u}, \underline{u}$ always have some underlying continuity:

\begin{proposition}\label{p:semicont}
Let $E\subseteq\R^d$ be a set with upper and lower boundaries $\overline{u}, \underline{u}$.  Then $\overline{u}, \underline{u}$ are upper/lower semicontinuous respectively.  
\end{proposition}
\begin{proof}
We show that $\overline{u}$ is upper semicontinuous.  Fix $x_0\in \R^{d-1}$, and let $x_n\to x_0$.   Without loss of generality, by passing to a subsequence we may suppose that $\lim\limits_{n\to \infty} \overline{u}(x_n)=L$.  Thus  as $(x_n, \overline{u}(x_n))\in \overline{E}$ for all $n$, we have that $(x_0,L)\in \overline{E}$ as well.  But then by the definition of $\overline{u}$, $\overline{u}(x_0)\geq L$.  Thus $\overline{u}(x_0)\geq \limsup\limits_{x\to x_0}\overline{u}(x)$, so it is upper semicontinuous.  
\end{proof}

A modulus of continuity can also be equivalently described on the level of sets as

\begin{definition}\label{d:setmodulusdefn}
Let $\omega: [0,\infty)\to [0,\infty)$ be a continuous function.  Then we say that an open (or closed) set $E\subseteq \R^d$ has modulus of continuity $\omega$ in the $x_d$-direction if for all $(z,z_d)\in \R^{d}$ with $z_d\geq \omega(|z|)$, 
\begin{equation}
E-(z,z_d)\subseteq E.
\end{equation}
\end{definition}

\begin{proposition}
Let $E\subseteq \R^d$ be an open set with upper and lower boundaries in the $x_d$-direction.  Then definitions \ref{d:boundarymodulusdefn} and \ref{d:setmodulusdefn} are equivalent.  
\end{proposition}

\begin{proof}
Assume $\overline{u}(x)-\underline{u}(y) \leq \omega(|x-y|)$ for all $x,y$, and fix some $Z\in \R^{d}$ with $z_d\geq \omega(|z|)$ and point $(x,x_d)\in E$.  Then by our assumption and the definition of $\overline{u}$ we have that 
\begin{equation}
x_d < \overline{u}(x) \leq \underline{u}(x-z) + \omega(|z|)\leq \underline{u}(x-z)+z_d.
\end{equation}
As $x_d-z_d<\underline{u}(x-z)$, we thus have by the definition of $\underline{u}$ that $(x-z, x_d-z_d)\in E$.  As $(x,x_d)\in E$ was arbitrary, we have that $E-(z,z_d)\subseteq E$.  Thus $E$ has modulus $\omega$ in the sense of Definition \ref{d:setmodulusdefn}.    

Conversely, suppose that $\overline{u}(x_0)-\underline{u}(y_0) = \omega(|x_0-y_0|)+\epsilon$ for some $x_0,y_0\in \R^{d-1}$ and $\epsilon>0$.  
As $(x_0, \overline{u}(x_0))\in \partial E$ and $\omega$ is continuous, we can find a point $(x,x_d)\in E$ such that $|\overline{u}(x_0)-x_d|<\epsilon/2$ and $|\omega(|x-y_0|)-\omega(|x_0-y_0|)| < \epsilon/2$.  Taking $Z = (x-y_0, x_d-\underline{u}(y_0))$, we have that 
\begin{equation}
x_d-\underline{u}(y_0)>\overline{u}(x_0)-\underline{u}(y_0)-\epsilon/2 = \omega(|x_0-y_0|)+\epsilon/2 > \omega(|x-y_0|), 
\end{equation}
but that $(x,x_d)-(x-y_0, x_d-\underline{u}(y_0)) = (y_0,\underline{u}(y_0))\not\in E$.  Thus $E$ does not have modulus $\omega$ in the sense of Definition \ref{d:setmodulusdefn}.  

\end{proof}

For most of our purposes, we'll be thinking about moduli of continuity in terms of the upper and lower boundaries, but the set definition works particularly will with the comparison principle and provides the easiest way to rigorously prove propagation of moduli for viscosity solutions.

\begin{proposition}\label{p:propagation}
Let $t\to E_t$ be the minimal viscosity supersolution .  Then $E_t$  has modulus $\omega$ for all time $t$.  
\end{proposition}

\begin{proof}
The proof follows by translation invariance and the comparison principle for the minimal viscosity supersolution, which we prove in Proposition \ref{p:comparisonprinciple} in the appendix.  

Fix $Z=(z,z_d)\in \R^{d}$ with $z_d\geq \omega(|z|)$, and let $E_t(Z) = E_t +Z$.  Then by translation invariance, $t\to E_t(Z)$ is the minimal viscosity supersolution of fractional mean curvature flow for the initial data $E_0(Z)$.

As $E_0\subseteq E_0(Z)$ by Definition \ref{d:setmodulusdefn}, it follows by the comparison principle that $E_t \subseteq E_t(Z)$ for all $t\geq 0$.  Since $Z\in \R^{d}$ with $z_d\geq \omega(|z|)$ was arbitrary, we thus have that $E_t$ has modulus $\omega$ for all times $t$.  
\end{proof}

Thus as in the graphical case, comparison principle and translation invariance imply that any modulus of continuity is propagated.  
Rather than just propagation though, our goal is to show an {\emph{improvement}} in our modulus of continuity.  

We now turn our attention to Theorem \ref{t:main}.  There we assume that our initial set $E_0$ is bounded between two Lipschitz subgraphs.  Since our plan is to describe an improvement in the modulus of continuity of $E_t$, it will be more convenient for us to translate this assumption into one about the modulus of continuity of $E_0$ directly.  Luckily, these are equivalent notions.

\begin{proposition}
Let $E\subseteq \R^d$ be an open set.  Then $E$ has modulus of continuity $\omega(r) = R+Lr$ in the $x_d$-direction if and only if there exists a Lipschitz function $u\in \dot{W}^{1,\infty}(\R^{d-1})$ with $||\nabla u||_{L^\infty}\leq L$ and 
\begin{equation}\label{e:usetinclusion}
\left\{(x,x_d)\bigg| x\in\R^{d-1}, x_d< u(x)-\frac{R}{2}\right\} \subseteq E\subseteq \left\{(x,x_d)\bigg| x\in\R^{d-1}, x_d < u(x)+\frac{R}{2}\right\},
\end{equation}
\end{proposition}

\begin{proof}
To begin, suppose that $E$ has modulus $\omega(r)=R+Lr$.  Then define $u:\R^{d-1}\to \R$ by 
\begin{equation}
u(x) = \frac{R}{2}+\inf\limits_{y\in \R^{d-1}} \underline{u}(y) + L|x-y|.
\end{equation}
Note that $u$ is a well defined function, as for any $y\in \R^{d-1}$
\begin{equation}
\underline{u}(y)+L|x-y| \geq \overline{u}(x)-\omega(|x-y|)+L|x-y| \geq \overline{u}(x)-R.
\end{equation}
Hence $u(x)$ exists with
\begin{equation}
u(x)+\frac{R}{2}\geq \overline{u}(x), 
\end{equation}
so
\begin{equation}
E\subseteq \{(x,x_d)| x_d<\overline{u}(x)\}\subseteq \left\{(x,x_d)\bigg| x_d<u(x)+\frac{R}{2}\right\}.
\end{equation}
As we also have by construction that $u(x)-\displaystyle\frac{R}{2}\leq \underline{u}(x)$, we also get the reverse inclusion 
\begin{equation}
\left\{(x,x_d)\bigg| x_d < u(x)-\frac{R}{2}\right\}\subseteq E.
\end{equation}
Finally as $u$ is the infimum of $L$-Lipschitz functions, it follows that $u$ is $L$-Lipschitz itself.  Thus we've shown the first direction.  

For the converse, suppose that $u\in \dot{W}^{1,\infty}(\R^{d-1})$ with $||\nabla u||_{L^\infty}\leq L$ satisfies \eqref{e:usetinclusion}.  Then since $u$ is continuous, it follows that 
\begin{equation}
\overline{E}\subseteq \left\{(x,x_d)\bigg| x_d\leq u(x)+\frac{R}{2}\right\}, \qquad \left\{(x,x_d)\bigg| x_d\geq u(x)-\frac{R}{2}\right\}\subseteq \mathcal{C}E.  
\end{equation}
Hence, $\overline{u}, \underline{u}$ are well defined with $\overline{u}\leq u+\displaystyle\frac{R}{2}$ and $\underline{u}\geq u-\displaystyle\frac{R}{2}$.  Thus for any $x,y\in \R^{d-1}$, 
\begin{equation}
\overline{u}(x)-\underline{u}(y) \leq u(x)-u(y)+R \leq R+L|x-y|.
\end{equation}
Thus $E$ has modulus of continuity $\omega(r) = R+L|x-y|$.  
\end{proof}

To make our strategy for the proof of Theorem \ref{t:main} more clear and understandable, let's first consider the case that we have a smooth, open initial data $E_0$ which has modulus
\begin{equation}
\omega(r) = R+Lr.
\end{equation}
and the flow $t\to E_t$ is smooth and exists for all time $t$. By rescaling the flow $t\to \displaystyle\frac{1}{R}E_{R^{1+s}t}$, we can assume without loss of generality that $R=1$.  Our goal then is to find a time dependent family of moduli of continuity $\omega:[0,\infty)\times [0,\infty)\to [0, \infty)$ satisfying

\begin{equation}\label{e:omegaassumptions}
\left\{\begin{array}{ll}
1). & \omega(0,r)>1+Lr \text{ for all } r \text{ and }\omega(t,r)>1+Lr \text{ for all }r>2 \text{ and times }t\in [0,T)
\\2). & \omega(t,\cdot) \text{ is } C^{1,1}, 0\leq \partial_r \omega(t,\cdot)\leq 1+L, \text{ with } \omega(t,0)>0, \text{ and } \partial_r \omega(t,0)=0  \text{ for all }t\in [0,T),
\\ 3). & \omega(T,\cdot) \text{ is }(1+L)\text{-Lipschitz with } \omega(T,0)=0, \end{array}\right.
\end{equation}
for some time $T$.  Then proving $E_t$ has modulus $\omega(t,\cdot)$ for all $t\in [0,T]$ would prove the Theorem \ref{t:main} for smooth flows.


\section{Breakthrough Argument}

Let $\omega(t,r)$ be a time dependent family of moduli of continuity satisfying out assumptions \eqref{e:omegaassumptions}.  

Let $E_0\subseteq \R^d$ be a smooth open set satisfying the modulus $1+Lr$, 
and assume the flow $t\to E_t$ is smooth and exists for all time $t$. 
Assume additionally that $E_0$ is flat at infinity, so 
\begin{equation}
\partial E_0 \setminus (B_M^{d-1}\times \R) = \{(x,0): |x|\geq M\},
\end{equation}
for some $M>0$.  Letting $\overline{u}(t,\cdot), \underline{u}(t,\cdot)$ denote the upper and lower boundaries of $E_t$, it then follows by Proposition \ref{p:limitlemma} that 
\begin{equation}\label{e:decay}
\lim\limits_{|x|\to \infty}\overline{u}(t,x)=\lim\limits_{|x|\to \infty}\underline{u}(t,x) = 0, \qquad \text{ uniformly for }t\in [0,T].
\end{equation}  
Furthermore by the Proposition \ref{p:propagation}, $E_t$ will have modulus $1+Lr$ for all times $t$.  

By assumption 1). in \eqref{e:omegaassumptions}, we know automatically that $E_0$ has modulus $\omega(0,\cdot)$.  Since the flow is smooth, it follows by continuity that $E_t$ will have modulus $\omega(t,\cdot)$ for sufficiently small times $t$.  

Suppose that $E_t$ loses the modulus $\omega(t,\cdot)$ before time $T$.  Let 
\begin{equation}\label{e:t_0defn}
t_0 = \sup\{t\in [0,T]: E_t \text{ has modulus } \omega(t,\cdot)\}.
\end{equation}
Then since have the modulus of continuity $\omega(t,\cdot)$ is a closed condition, we know that $E_{t_0}$ has modulus $\omega(t_0, \cdot)$.  

Suppose that $E_{t_0}$ strictly had the modulus $\omega$.  That is, for any $x,y\in \R^{d-1}$ 
\begin{equation}
\overline{u}(t_0,x)-\underline{u}(t_0, y) < \omega(t_0, |x-y|).
\end{equation}
We will show that in this case that $E_{t_0+\epsilon}$ has the modulus $\omega(t_0+\epsilon, \cdot)$ for $\epsilon$ sufficiently small, contradicting the definition of $t_0$.  

Let $\delta = \min\{\omega(t,0): 0\leq t\leq \frac{t_0+T}{2}\}$.  
By \eqref{e:decay} we have that there is some $R>0$ such that for $|x|,|y|> R$,
\begin{equation}
\overline{u}(t_0+\epsilon,x) - \delta/3<0 < \underline{u}(t_0+\epsilon,y)+\delta/3,
\end{equation}
for any $0\leq \epsilon \leq \frac{T-t_0}{2}$.  

Now suppose that $|x|>R+2$.  Then for any $y\in \R^{d-1}$, we have that either $|y|>R$ or $|x-y|>2$.  If $|x-y|>2$, then by assumption 1). of \eqref{e:omegaassumptions}
\begin{equation}
\overline{u}(t_0+\epsilon,x)-\underline{u}(t_0+\epsilon,y)\leq 1+L|x-y|< \omega(t_0+\epsilon,|x-y|),
\end{equation}
for any $\epsilon$.  If $|y|>R$, then similarly we have for any $0\leq \epsilon\leq \frac{T-t_0}{2}$
\begin{equation}
\overline{u}(t_0+\epsilon,x)-\underline{u}(t_0+\epsilon, y) <2\delta/3 < \omega(t_0+\epsilon,0)\leq \omega(t_0+\epsilon, |x-y|).
\end{equation}

A symmetric argument clearly works in the case that $|y|>R+2$.  Thus the only thing that remains to show that $E_{t_0+\epsilon}$ has modulus $\omega(t_0+\epsilon,\cdot)$ is to consider the case when both $|x|, |y|\leq R+2$.  


We know from Proposition \eqref{p:semicont} that $\overline{u}(t,\cdot), \underline{u}(t,\cdot)$ are upper/lower semicontinuous in space respectively.  Since by assumption the flow $t\to E_t$ is smooth, they will be semicontinuous in time as well.   Thus by uniform semicontinuity, 
\begin{equation}
\overline{u}(t_0+\epsilon,x)-\underline{u}(t_0+\epsilon,y)<\omega(t_0+\epsilon,|x-y|),
\end{equation}
for $\epsilon$ sufficiently small when $|x|, |y|<R+2$.  Hence, $E_{t_0+\epsilon}$ has modulus $\omega(t_0+\epsilon, \cdot)$ for $\epsilon$ sufficiently small, violating the definition of $t_0$ \eqref{e:t_0defn}.

Thus if the set $E_t$ was to lose the modulus $\omega(t, \cdot)$ before time $T$, then necessarily there must be two points $x,y\in \R^{d-1}$ such that 
\begin{equation}
\overline{u}(t_0,x) - \underline{u}(t_0,y) = \omega(t_0, |x-y|).
\end{equation}
We will show in sections 3 and 4 that in this case for the right choice of family $\omega(t,r)$, 
\begin{equation}
\partial_t \overline{u}(t_0,x)-\partial_t \underline{u}(t_0,y) < \partial_t \omega(t_0, |x-y|),
\end{equation}
contradicting the fact that $E_t$ had the modulus before time $t_0$.  

%


\section{Curvature Estimates}

Everything from now on will be at a fixed time $t_0\in (0,T)$, so we will simply suppress the time variable.  Our standing assumption is that the open set $E=E_{t_0}$ has some modulus $\omega(\cdot)$ satisfying 
\begin{equation}\label{e:omegabound}
\left\{\begin{array} {ll}1). &\omega(0)>0 \text{ and }\omega(r)>1+Lr \text{ for } r>2,
\\ 2). & \omega(\cdot) \text{ is } C^{1,1} \text{ with } 0\leq \omega'(\cdot) \leq (1+L) \text{ and }\omega'(0)=0,
\\ 3). & \overline{u}(x)-\underline{u}(y)\leq \omega(|x-y|), \quad \forall x,y\in \R^{d-1},
\\4). & \overline{u}(\frac{\xi}{2})-\underline{u}(\frac{-\xi}{2}) = \omega(|\xi|),\end{array}\right. 
\end{equation}
for some $\xi\in \R^{d-1}$ with $|\xi|\leq 2$ .

Our goal is to use our assumptions \eqref{e:omegabound} and the equation \eqref{e:meancurvfloweqnset} to bound the difference between $\partial_t\overline{u}(\xi/2)-\partial_t\underline{u}(-\xi/2)$ from above in terms of $\omega$.  

To begin, we're first going to derive the proper equation for $\overline{u},\underline{u}$.  Note that they are respectively upper and lower semicontinuous, and locally smooth (since $E$ is smooth) whenever the outward unit normal doesn't lie in the $\R^{d-1}$ plane.   As $\overline{u}$ is touched from above by $\omega$ and $\underline{u}$ is touched from below by $-\omega$, this is necessarily the case so we thus have that 
\begin{equation}
\nabla \overline{u}\left(\frac{\xi}{2}\right) = \nabla \underline{u}\left(\frac{-\xi}{2}\right) =  \omega'(|\xi|) \frac{\xi}{|\xi|}.
\end{equation}
Note that as $\omega'(0)=0$ by 2) in \eqref{e:omegabound}, this is still valid when $\xi=0$.  

The point $(\frac{\xi}{2},\overline{u}(\frac{\xi}{2}))\in \partial E$ thus has outward unit normal $\displaystyle\frac{( -\omega'(|\xi|)\frac{\xi}{|\xi|}, 1)}{\sqrt{1+\omega'(|\xi|)^2}}$, so from the fractional mean curvature flow equation \eqref{e:meancurvfloweqnset} we get that 

\begin{equation}
\partial_t \overline{u}\left(\frac{\xi}{2}\right) = s(1-s)\sqrt{1+ \omega'(|\xi|)^2} P.V. \int\limits_{\R^{d-1}}\int\limits_\R \frac{\1_E^\pm\left(\frac{\xi}{2}+z, \overline{u}\left(\frac{\xi}{2}\right)+z_d\right)}{(|z|^2+z_d^2)^{(d+s)/2}}dz_d dz,
\end{equation}
where $\1_E^{\pm}(X) = \1_E(X) -\1_{\mathcal{C}E}(X)$ is the signed characteristic function.  Similarly,  
\begin{equation}
\partial_t \underline{u}\left(\frac{-\xi}{2}\right) = s(1-s)\sqrt{1+ \omega'(|\xi|)^2}  P.V.\int\limits_{\R^{d-1}}\int\limits_\R \frac{\1^\pm_E\left(\frac{-\xi}{2}+z, \underline{u}\left(\frac{-\xi}{2}\right)+z_d\right)}{(|z|^2+z_d^2)^{(d+s)/2}}dz_d dz.
\end{equation}

Taking the difference and moving like constants to the other side, we thus have that 
\begin{equation}\label{e:preciseboundarydiff}
\frac{\partial_t \overline{u}\left(\frac{\xi}{2}\right)-\partial_t \underline{u}\left(\frac{-\xi}{2}\right)}{s(1-s)\sqrt{1+\omega'(|\xi|)^2}} = P.V. \int\limits_{\R^{d-1}}\int\limits_\R \frac{\1^\pm_E\left(\frac{\xi}{2}+z, \overline{u}\left(\frac{\xi}{2}\right)+z_d\right)-\1^\pm_E\left(\frac{-\xi}{2}+z, \underline{u}\left(\frac{-\xi}{2}\right)+z_d\right)}{(|z|^2+z_d^2)^{(d+s)/2}}dz_d dz.
\end{equation}

\begin{remark}
Its important to note that here we are implicitly taking advantage of the fact that in the smooth case, $H_s(X,E) = H_s(X,\overline{E})$.  This will no longer be true in general for viscosity solutions, and that difference is the source of most of the extra technical difficulties in that regime.
\end{remark}

\begin{lemma}\label{l:fixedzestimate}
Let $E\subseteq \R^d$ be an open set with modulus $\omega$ satisfying \eqref{e:omegabound}.  Then 
\begin{equation}\label{e:nonnegative}
\1^\pm_E\left(\frac{\xi}{2}+z, \overline{u}\left(\frac{\xi}{2}\right)+z_d\right)-\1^\pm_E\left(\frac{-\xi}{2}+z, \underline{u}\left(\frac{-\xi}{2}\right)+z_d\right)\leq 0. 
\end{equation}
Furthermore, for any $0\not = z\in \R^{d-1}$, 
\begin{equation}\label{e:firstbound}
\begin{split}
\int\limits_\R &\frac{\1^\pm_E\left(\frac{\xi}{2}+z, \overline{u}\left(\frac{\xi}{2}\right)+z_d\right)-\1^\pm_E\left(\frac{-\xi}{2}+z, \underline{u}\left(\frac{-\xi}{2}\right)+z_d\right)}{(|z|^2+z_d^2)^{(d+s)/2}}dz_d 
\\ &\qquad  \leq \int\limits_{\underline{u}\left(\frac{\xi}{2}+z\right)-\overline{u}\left(\frac{\xi}{2}\right)}^{\overline{u}\left(\frac{-\xi}{2}+z\right)-\underline{u}\left(\frac{-\xi}{2}\right)} \frac{\1^\pm_E\left(\frac{\xi}{2}+z, \overline{u}\left(\frac{\xi}{2}\right)+z_d\right)-\1^\pm_E\left(\frac{-\xi}{2}+z, \underline{u}\left(\frac{-\xi}{2}\right)+z_d\right)}{(|z|^2+\omega(|z|)^2)^{(d+s)/2}}dz_d
\end{split}
\end{equation}
\end{lemma}

\begin{proof}
To begin, note that the fact that $E$ has modulus $\omega$ and 4). in \eqref{e:omegabound} immediately implies
\begin{equation}
E-\left(\frac{\xi}{2}, \overline{u}\left(\frac{\xi}{2}\right)\right)\subseteq E+\left(\xi-\frac{\xi}{2}, \omega(|\xi|)-\overline{u}\left(\frac{\xi}{2}\right)\right) = E -\left(\frac{-\xi}{2}, \underline{u}\left(\frac{-\xi}{2}\right)\right), 
\end{equation}
and hence \eqref{e:nonnegative}.  

This is just a reflection of Proposition \ref{p:propagation}.  In order to turn this into a quantitative statement, we'll have to rely on the definitions of $\overline{u}, \underline{u}$ ( see Definition \ref{d:upperlowerboundary} ) and assumption 3). of \eqref{e:omegabound}.   

Now fix some $z\not = 0$.  By \eqref{e:nonnegative}, we get immediately that if $\1^\pm_E\left(\frac{\xi}{2}+z, \overline{u}\left(\frac{\xi}{2}\right)+z_d\right)=1$ then $1^\pm_E\left(\frac{-\xi}{2}+z, \underline{u}\left(\frac{-\xi}{2}\right)+z_d\right)=1$, and if $\1^\pm_E\left(\frac{-\xi}{2}+z, \underline{u}\left(\frac{-\xi}{2}\right)+z_d\right)=-1$ then $\1^\pm_E\left(\frac{\xi}{2}+z, \overline{u}\left(\frac{\xi}{2}\right)+z_d\right)=-1$.  In particular, by the definition of $\underline{u}\left(\frac{\xi}{2}+z\right), \overline{u}\left(\frac{-\xi}{2}+z\right)$ this implies that  
\begin{equation}\label{e:0bound}
\begin{split}
\1^\pm_E\left(\frac{\xi}{2}+z, \overline{u}\left(\frac{\xi}{2}\right)+z_d\right)-\1^\pm_E\left(\frac{-\xi}{2}+z, \underline{u}\left(\frac{-\xi}{2}\right)+z_d\right)=0
\\ \text{ for } z_d< \underline{u}\left(\frac{\xi}{2}+z\right)-\overline{u}\left(\frac{\xi}{2}\right) \text{ or } z_d\geq \overline{u}\left(\frac{-\xi}{2}+z\right)-\underline{u}\left(\frac{-\xi}{2}\right). 
\end{split}
\end{equation}
As $\underline{u}\left(\frac{\xi}{2}+z\right)-\overline{u}\left(\frac{\xi}{2}\right)\geq -\omega(|z|)$ and $\overline{u}\left(\frac{-\xi}{2}+z\right)-\underline{u}\left(\frac{-\xi}{2}\right)\leq \omega(|z|)$ by 3). in \eqref{e:omegabound}, combining \eqref{e:nonnegative} and \eqref{e:0bound} gives us 
\begin{equation}
\begin{split}
\int\limits_\R &\frac{\1^\pm_E\left(\frac{\xi}{2}+z, \overline{u}\left(\frac{\xi}{2}\right)+z_d\right)-\1^\pm_E\left(\frac{-\xi}{2}+z, \underline{u}\left(\frac{-\xi}{2}\right)+z_d\right)}{(|z|^2+z_d^2)^{(d+s)/2}}dz_d 
\\&= \int\limits_{\underline{u}\left(\frac{\xi}{2}+z\right)-\overline{u}\left(\frac{\xi}{2}\right)}^{\overline{u}\left(\frac{-\xi}{2}+z\right)-\underline{u}\left(\frac{-\xi}{2}\right)} \frac{\1^\pm_E\left(\frac{\xi}{2}+z, \overline{u}\left(\frac{\xi}{2}\right)+z_d\right)-\1^\pm_E\left(\frac{-\xi}{2}+z, \underline{u}\left(\frac{-\xi}{2}\right)+z_d\right)}{(|z|^2+z_d^2)^{(d+s)/2}}dz_d
\\&\leq\int\limits_{\underline{u}\left(\frac{\xi}{2}+z\right)-\overline{u}\left(\frac{\xi}{2}\right)}^{\overline{u}\left(\frac{-\xi}{2}+z\right)-\underline{u}\left(\frac{-\xi}{2}\right)} \frac{\1^\pm_E\left(\frac{\xi}{2}+z, \overline{u}\left(\frac{\xi}{2}\right)+z_d\right)-\1^\pm_E\left(\frac{-\xi}{2}+z, \underline{u}\left(\frac{-\xi}{2}\right)+z_d\right)}{(|z|^2+\omega(|z|)^2)^{(d+s)/2}}dz_d.
\end{split}
\end{equation}

\end{proof}

\begin{lemma} \label{l:P.V.estimate}
Let $E\subseteq \R^d$ be an open set with modulus $\omega$ satisfying \eqref{e:omegabound}.  Then
\begin{equation}\label{e:mainintest}
\begin{split}
P.&V. \int\limits_{\R^{d-1}}\int\limits_\R \frac{\1^\pm_E\left(\frac{\xi}{2}+z, \overline{u}\left(\frac{\xi}{2}\right)+z_d\right)-\1^\pm_E\left(\frac{-\xi}{2}+z, \underline{u}\left(\frac{-\xi}{2}\right)+z_d\right)}{(|z|^2+z_d^2)^{(d+s)/2}}dzdz_d
\\ &\leq\frac{-2}{(3(1+L))^{d+s}} \left(\int\limits_{\R^{d-1}} \frac{\omega(|\xi|)-\left(\overline{u}\left(\frac{\xi}{2}+z\right)-\underline{u}\left(\frac{-\xi}{2}+z\right)\right)}{(|z|^2+\omega(0)^2)^{(d+s)/2}}dz
+\int\limits_{\R^{d-1}}\frac{\overline{u}(z)-\underline{u}(z)}{\max\limits_{\pm}(|z\pm \frac{\xi}{2}|^2+\omega(0)^2)^{(d+s)/2}}dz\right).
\end{split}
\end{equation}
\end{lemma}

\begin{proof}
By Lemma \ref{l:fixedzestimate} it suffices to show that 
\begin{equation}
\begin{split}
 \int\limits_{\R^{d-1}}&\left(\int\limits_{\underline{u}\left(\frac{\xi}{2}+z\right)-\overline{u}\left(\frac{\xi}{2}\right)}^{\overline{u}\left(\frac{-\xi}{2}+z\right)-\underline{u}\left(\frac{-\xi}{2}\right)} \frac{\1^\pm_E\left(\frac{\xi}{2}+z, \overline{u}\left(\frac{\xi}{2}\right)+z_d\right)-\1^\pm_E\left(\frac{-\xi}{2}+z, \underline{u}\left(\frac{-\xi}{2}\right)+z_d\right)}{(|z|^2+\omega(|z|)^2)^{(d+s)/2}}dz_d \right)dz
\\ &\leq\frac{-2}{(3(1+L))^{d+s}} \left(\int\limits_{\R^{d-1}} \frac{\omega(|\xi|)-\left(\overline{u}\left(\frac{\xi}{2}+z\right)-\underline{u}\left(\frac{-\xi}{2}+z\right)\right)}{(|z|^2+\omega(0)^2)^{(d+s)/2}}dz
+\int\limits_{\R^{d-1}}\frac{\overline{u}(z)-\underline{u}(z)}{\max\limits_{\pm}(|z\pm \frac{\xi}{2}|^2+\omega(0)^2)^{(d+s)/2}}dz\right).
\end{split}
\end{equation}

Fix some $z\not = 0$.  Using the definitions of $\overline{u}, \underline{u}$, we can more precisely bound   
\begin{equation}\label{e:fixedz'}
\int\limits_{\underline{u}\left(\frac{\xi}{2}+z\right)-\overline{u}\left(\frac{\xi}{2}\right)}^{\overline{u}\left(\frac{-\xi}{2}+z\right)-\underline{u}\left(\frac{-\xi}{2}\right)} \frac{\1^\pm_E\left(\frac{\xi}{2}+z, \overline{u}\left(\frac{\xi}{2}\right)+z_d\right)-\1^\pm_E\left(\frac{-\xi}{2}+z, \underline{u}\left(\frac{-\xi}{2}\right)+z_d\right)}{(|z|^2+\omega(|z|)^2)^{(d+s)/2}}dz_d,
\end{equation}
using that 
\begin{equation}
\begin{split}
\overline{u}\left(\frac{\xi}{2}+z\right)&-\overline{u}\left(\frac{\xi}{2}\right)\leq z_d< \underline{u}\left(\frac{-\xi}{2}+z\right)-\underline{u}\left(\frac{-\xi}{2}\right) 
\\ & \Rightarrow \1^\pm_E\left(\frac{\xi}{2}+z, \overline{u}\left(\frac{\xi}{2}\right)+z_d\right)-\1^\pm_E\left(\frac{-\xi}{2}+z, \underline{u}\left(\frac{-\xi}{2}\right)+z_d\right) = -2.
\end{split}
\end{equation}

Thus we can bound \eqref{e:fixedz'} by 
\begin{equation}\label{e:secondbound}
\begin{split}
\int\limits_{\underline{u}\left(\frac{\xi}{2}+z\right)-\overline{u}\left(\frac{\xi}{2}\right)}^{\overline{u}\left(\frac{-\xi}{2}+z\right)-\underline{u}\left(\frac{-\xi}{2}\right)}  &\frac{\1^\pm_E\left(\frac{\xi}{2}+z, \overline{u}\left(\frac{\xi}{2}\right)+z_d\right)-\1^\pm_E\left(\frac{-\xi}{2}+z, \underline{u}\left(\frac{-\xi}{2}\right)+z_d\right)}{(|z|^2+\omega(|z|)^2)^{(d+s)/2}}dz_d 
\\ &\leq \int\limits_{\overline{u}\left(\frac{\xi}{2}+z\right)-\overline{u}\left(\frac{\xi}{2}\right)}^{\underline{u}\left(\frac{-\xi}{2}+z\right)-\underline{u}\left(\frac{-\xi}{2}\right)} \frac{-2}{(|z|^2+\omega(|z|)^2)^{(d+s)/2}}dz_d
\\ &\quad + \int\limits_{\underline{u}\left(\frac{\xi}{2}+z\right)- \overline{u}\left(\frac{\xi}{2}\right)}^{\overline{u}\left(\frac{\xi}{2}+z\right)-\overline{u}\left(\frac{\xi}{2}\right)}  \frac{\1^\pm_E\left(\frac{\xi}{2}+z, \overline{u}\left(\frac{\xi}{2}\right)+z_d\right)-\1^\pm_E\left(\frac{-\xi}{2}+z, \underline{u}\left(\frac{-\xi}{2}\right)+z_d\right)}{(|z|^2+\omega(|z|)^2)^{(d+s)/2}}dz_d  
\\ &\quad +  \int\limits_{\underline{u}\left(\frac{-\xi}{2}+z\right)-\underline{u}\left(\frac{-\xi}{2}\right)}^{\overline{u}\left(\frac{-\xi}{2}+z\right)-\underline{u}\left(\frac{-\xi}{2}\right)}  \frac{\1^\pm_E\left(\frac{\xi}{2}+z, \overline{u}\left(\frac{\xi}{2}\right)+z_d\right)-\1^\pm_E\left(\frac{-\xi}{2}+z, \underline{u}\left(\frac{-\xi}{2}\right)+z_d\right)}{(|z|^2+\omega(|z|)^2)^{(d+s)/2}}dz_d.
\end{split}
\end{equation}

Using 4). in \eqref{e:omegabound} we can rewrite the first integral on the right hand side of \eqref{e:secondbound} as 
\begin{equation}\label{e:modcontdiff}
\int\limits_{\overline{u}\left(\frac{\xi}{2}+z\right)-\overline{u}\left(\frac{\xi}{2}\right)}^{\underline{u}\left(\frac{-\xi}{2}+z\right)-\underline{u}\left(\frac{-\xi}{2}\right)}\frac{-2}{(|z|^2+\omega(|z|)^2)^{(d+s)/2}}dz_d
= (-2)\frac{\omega(|\xi|)-\left(\overline{u}\left(\frac{\xi}{2}+z\right)-\underline{u}\left(\frac{-\xi}{2}+z\right)\right)}{(|z|^2+\omega(|z|)^2)^{(d+s)/2}}.
\end{equation}

As for the other two integrals on the right hand side of \eqref{e:secondbound}, by translating our $z_d$ bounds and using that $\overline{u}\left(\frac{\xi}{2}\right)-\underline{u}\left(\frac{-\xi}{2}\right)=\omega(|\xi|)$ we get that
\begin{equation}\label{e:+xiint}
\begin{split}
\displaystyle\int\limits_{\underline{u}\left(\frac{\xi}{2}+z\right)- \overline{u}\left(\frac{\xi}{2}\right)}^{\overline{u}\left(\frac{\xi}{2}+z\right)-\overline{u}\left(\frac{\xi}{2}\right)} \1^\pm_E\left(\frac{\xi}{2}+z, \overline{u}\left(\frac{\xi}{2}\right)+z_d\right)&-\1^\pm_E\left(\frac{-\xi}{2}+z, \underline{u}\left(\frac{-\xi}{2}\right)+z_d\right)dz_d
\\& = \int\limits_{\underline{u}\left(\frac{\xi}{2}+z\right)}^{\overline{u}\left(\frac{\xi}{2}+z\right)}  \1^\pm_E\left(\frac{\xi}{2}+z, z_d\right)-\1^\pm_E\left(\frac{-\xi}{2}+z, z_d-\omega(|\xi|)\right)dz_d
\\&= \int\limits_{\underline{u}\left(\frac{\xi}{2}+z\right)}^{\overline{u}\left(\frac{\xi}{2}+z\right)}  \1^\pm_E\left(\frac{\xi}{2}+z, z_d\right)-1dz_d.
\end{split}
\end{equation}
To go from line 2 to line 3, note that $z_d<\overline{u}\left(\frac{\xi}{2}+z\right)\leq \underline{u}\left(\frac{-\xi}{2}+z\right)+\omega(|\xi|)$.  Hence, $\1^\pm_E\left(\frac{-\xi}{2}+z, z_d-\omega(|\xi|)\right)=1$.  

By a symmetric argument, 
\begin{equation}\label{e:-xiint}
\begin{split}
\displaystyle\int\limits_{\underline{u}\left(\frac{-\xi}{2}+z\right)- \overline{u}\left(\frac{-\xi}{2}\right)}^{\overline{u}\left(\frac{-\xi}{2}+z\right)-\overline{u}\left(\frac{-\xi}{2}\right)}  \1^\pm_E\left(\frac{\xi}{2}+z, \overline{u}\left(\frac{\xi}{2}\right)+z_d\right)&-\1^\pm_E\left(\frac{-\xi}{2}+z, \underline{u}\left(\frac{-\xi}{2}\right)+z_d\right)dz_d
\\&= \int\limits_{\underline{u}\left(\frac{-\xi}{2}+z\right)}^{\overline{u}\left(\frac{-\xi}{2}+z\right)}  -1 - \1^\pm_E\left(\frac{-\xi}{2}+z, z_d\right)dz_d.
\end{split}
\end{equation}

Shifting the value of $z$ in \eqref{e:+xiint} and \eqref{e:-xiint} by $\mp \xi/2$ and adding them together thus gives that 
\begin{equation}\label{e:upperlowerdiff}
\begin{split}
\int\limits_{\underline{u}(z)}^{\overline{u}(z)}&\frac{\1^\pm_E(z, z_d)-1}{(|z-\frac{\xi}{2}|^2+\omega(|z-\frac{\xi}{2}|)^2)^{(d+s)/2}} + \frac{-1-\1^\pm_E(z, z_d)}{(|z+\frac{\xi}{2}|^2+\omega(|z+\frac{\xi}{2}|)^2)^{(d+s)/2}} dz_d
\\ &\leq \int\limits_{\underline{u}(z)}^{\overline{u}(z)} \frac{-2}{\max\limits_{\pm}(|z\pm\frac{\xi}{2}|^2+\omega(|z\pm \frac{\xi}{2}|)^2)^{(d+s)/2}}
 = (-2)\frac{\overline{u}(z)-\underline{u}(z)}{\max\limits_{\pm}(|z\pm\frac{\xi}{2}|^2+\omega(|z\pm \frac{\xi}{2}|)^2)^{(d+s)/2}}
\end{split}
\end{equation}

Plugging in \eqref{e:upperlowerdiff} and \eqref{e:modcontdiff} into \eqref{e:secondbound} and integrating in $z$ thus gives us that 
\begin{equation}
\begin{split}
P.V. &\int\limits_{\R^{d-1}}\int\limits_\R \frac{\1^\pm_E\left(\frac{\xi}{2}+z, \overline{u}\left(\frac{\xi}{2}\right)+z_d\right)-\1^\pm_E\left(\frac{-\xi}{2}+z, \underline{u}\left(\frac{-\xi}{2}\right)+z_d\right)}{(|z|^2+z_d^2)^{(d+s)/2}}dzdz_d
\\ &\leq -2\int\limits_{\R^{d-1}} \frac{\omega(|\xi|)-\left(\overline{u}\left(\frac{\xi}{2}+z\right)-\underline{u}\left(\frac{-\xi}{2}+z\right)\right)}{(|z|^2+\omega(|z|)^2)^{(d+s)/2}}dz
 -2\int\limits_{\R^{d-1}}\frac{\overline{u}(z)-\underline{u}(z)}{\max\limits_{\pm}(|z\pm\frac{\xi}{2}|^2+\omega(|z\pm\frac{\xi}{2}|)^2)^{(d+s)/2}}dz
\end{split}
\end{equation}
Finally, the proof is complete using that $\omega$ is $(1+L)$-Lipschtiz by assumption 3). in \eqref{e:omegabound}, so 
\begin{equation}
\begin{split}
\frac{1}{(|z|^2+\omega(|z|)^2)^{(d+s)/2}}&\geq \frac{1}{(|z|^2+((1+L)|z|+\omega(0))^2)^{(d+s)/2}}\geq  \frac{1}{((2(1+L)^2+1)|z|^2+2\omega(0)^2)^{(d+s)/2}}
\\& \geq \frac{1}{(3(1+L))^{d+s}} \frac{1}{(|z|^2+\omega(0)^2)^{(d+s)/2}}.
\end{split}
\end{equation}
\end{proof}

Combining \eqref{e:preciseboundarydiff} with Lemma \ref{l:P.V.estimate} thus gives us that 

\begin{equation}\label{e:timediffestimate}
\begin{split}
\partial_t \overline{u}\left(\frac{\xi}{2}\right)&-\partial_t \underline{u}\left(\frac{-\xi}{2}\right)
\\&\lesssim_{d,s,L}-\int\limits_{\R^{d-1}} \frac{\omega(|\xi|)-\left(\overline{u}\left(\frac{\xi}{2}+z\right)-\underline{u}\left(\frac{-\xi}{2}+z\right)\right)}{(|z|^2+\omega(0)^2)^{(d+s)/2}}dz
-\int\limits_{\R^{d-1}}\frac{\overline{u}(z)-\underline{u}(z)}{\max\limits_{\pm}(|z\pm \frac{\xi}{2}|^2+\omega(0)^2)^{(d+s)/2}}dz.
\end{split}
\end{equation}

In order to remove the dependence on $\overline{u},\underline{u}$ from \eqref{e:timediffestimate} and get an upper bound depending only on $\omega$, we now alter a useful integral rearrangement argument of \cite{KiselevIntRearrange} to get

\begin{lemma}\label{l:intrearrange}
Let $E$ be an open set with modulus $\omega$ satisfying assumptions \eqref{e:omegabound}.  Then 
\begin{equation}
\begin{split}
\int\limits_{\R^{d-1}} &\frac{\omega(|\xi|)-\left(\overline{u}\left(\frac{\xi}{2}+z\right)-\underline{u}\left(\frac{-\xi}{2}+z\right)\right)}{(|z|^2+\omega(0)^2)^{(d+s)/2}}dz
+\int\limits_{\R^{d-1}}\frac{\overline{u}(z)-\underline{u}(z)}{\max\limits_{\pm}(|z\pm \frac{\xi}{2}|^2+\omega(0)^2)^{(d+s)/2}}dz
\\ &\geq C(d,s)\left( \int\limits_0^{\frac{|\xi|}{2}} \frac{2\omega(|\xi|)-\omega(|\xi|+2\eta)-\omega(|\xi|-2\eta)}{(\eta^2+\omega(0)^2)^{(2+s)/2}}d\eta + \int\limits_{\frac{|\xi|}{2}}^\infty \frac{2\omega(|\xi|)+\omega(2\eta -|\xi|)-\omega(2\eta+|\xi|)}{(\eta^2+\omega(0)^2)^{(2+s)/2}}d\eta\right).
\end{split}
\end{equation}
where $C(d,s) =\displaystyle\frac{\Gamma\left(\frac{d-2}{2}\right)\Gamma\left(\frac{2+s}{2}\right)}{2\Gamma\left(\frac{d+s}{2}\right)}$ for $d\geq 3$, and $C(2,s)=1$. 
\end{lemma}

\begin{proof}
We prove Lemma \ref{l:intrearrange} in the case that $d\geq 3$.  The case when $d=2$ follows from a clear simplification.  

We write $z\in \R^{d-1}$ as $z = (\eta, \nu)\in \R\times \R^{d-2}$, and without loss of generality assume that $\xi = (|\xi|,0)$.

To begin, let 
\begin{equation}
K(\eta, \nu) = \frac{1}{(\eta^2+|\nu|^2+\omega(0)^2)^{(d+s)/2}}.
\end{equation}

As $K(\eta, \nu)$ is radial, we have that 
\begin{equation}
\begin{split}
\int\limits_{\R^{d-2}}\int\limits_{\R} &\left[ \omega(|\xi|)- \overline{u}\left(\frac{|\xi|}{2}+\eta,\nu\right)+\underline{u}\left(-\frac{|\xi|}{2}+\eta,\nu\right)\right] K(\eta,\nu)d\eta d\nu
\\ &= \int\limits_{\R^{d-2}}\int\limits_{\R}\left[ \omega(|\xi|)- \overline{u}\left(\frac{|\xi|}{2}+\eta,\nu\right)+\underline{u}\left(-\frac{|\xi|}{2}-\eta,\nu\right)\right] K(\eta,\nu)d\eta d\nu.
\end{split}
\end{equation}

Fix some $\nu\in \R^{d-1}$.  Then breaking up the $\eta$ integral, rearranging, and adding terms gives
\begin{equation}
\begin{split}
\int\limits_{\R}& \left[ \omega(|\xi|)- \overline{u}\left(\frac{|\xi|}{2}+\eta,\nu\right)+\underline{u}\left(-\frac{|\xi|}{2}-\eta,\nu\right)\right] K(\eta,\nu)d\eta 
\\& =\int\limits_{-\frac{|\xi|}{2}}^\infty \left[ \omega(|\xi|)- \overline{u}\left(\frac{|\xi|}{2}+\eta,\nu\right)+\underline{u}\left(-\frac{|\xi|}{2}-\eta,\nu\right)\right] K(\eta,\nu)d\eta 
\\& \quad +\int\limits_{-\frac{|\xi|}{2}}^\infty \left[ \omega(|\xi|)- \overline{u}\left(-\frac{|\xi|}{2}-\eta,\nu\right)+\underline{u}\left(\frac{|\xi|}{2}+\eta,\nu\right)\right] K(-|\xi|-\eta,\nu)d\eta 
\\&= \int\limits_{-\frac{|\xi|}{2}}^\infty  \omega(|\xi|)(K(\eta, \nu)+K(|\xi|+\eta,\nu)) 
\\& \qquad \qquad \qquad- \left[\overline{u}\left(\frac{|\xi|}{2}+\eta,\nu\right) -\underline{u}\left(-\frac{|\xi|}{2}-\eta,\nu\right)\right] (K(\eta,\nu)-K(|\xi|+\eta,\nu))d\eta 
\\& \quad - \int\limits_{-\frac{|\xi|}{2}}^\infty \left[\overline{u}\left(\frac{|\xi|}{2}+\eta,\nu\right)- \underline{u}\left(\frac{|\xi|}{2}+\eta,\nu\right) + \overline{u}\left(-\frac{|\xi|}{2}-\eta,\nu\right)- \underline{u}\left(\frac{-|\xi|}{2}-\eta,\nu\right)\right] K(|\xi|+\eta,\nu)d\eta 
\\& =\int\limits_{-\frac{|\xi|}{2}}^\infty  \omega(|\xi|)(K(\eta, \nu)+K(|\xi|+\eta,\nu)) 
\\& \qquad \qquad \qquad- \left[\overline{u}\left(\frac{|\xi|}{2}+\eta,\nu\right) -\underline{u}\left(-\frac{|\xi|}{2}-\eta,\nu\right)\right] (K(\eta,\nu)-K(|\xi|+\eta,\nu))d\eta 
\\& \quad - \int\limits_{\R} \left[\overline{u}(\eta, \nu)-\underline{u}(\eta, \nu)\right]\min\{K\left(\eta\pm \frac{|\xi|}{2},\nu\right)\}d\eta .
\end{split}
\end{equation}
The last integral is the only real difference between the argument we give and the one \cite{KiselevIntRearrange} gives, and it arises naturally out of the fact that $\partial E$ is not the graph of a function.  However, noting that 
\begin{equation}
\min\limits_{\pm} K\left(z\pm \frac{\xi}{2}\right) = \frac{1}{\max\limits_{\pm}(|z\pm \frac{\xi}{2}|^2+\omega(0)^2)^{(d+s)/2}},
\end{equation}
we see this term is precisely 
\begin{equation}
 \int\limits_{\R^{d-2}}\int\limits_{\R} \left[\overline{u}(\eta, \nu)-\underline{u}(\eta, \nu)\right]\min\{K\left(\eta\pm \frac{|\xi|}{2},\nu\right)\}d\eta d\nu
=\int\limits_{\R^{d-1}}\frac{\overline{u}(z)-\underline{u}(z)}{\max\limits_{\pm}(|z\pm \frac{\xi}{2}|^2+\omega(0)^2)^{(d+s)/2}}dz 
\end{equation}

Integrating in $\nu$, taking into account this cancelation, adding/subtracting $\omega(|\xi|+2\eta)$, and using that $K$ is nonincreasing, we get that 
\begin{equation}
\begin{split}
\int\limits_{\R^{d-2}}\int\limits_{-\frac{|\xi|}{2}}^\infty &\left[ \omega(|\xi|)- \overline{u}\left(\frac{|\xi|}{2}+\eta,\nu\right)+\underline{u}\left(-\frac{|\xi|}{2}+\eta,\nu\right)\right] K(\eta,\nu)d\eta d\nu 
\\& +\int\limits_{\R^{d-2}} \int\limits_{\R} \left[\overline{u}(\eta, \nu)-\underline{u}(\eta, \nu)\right]\min\{K\left(\eta\pm \frac{|\xi|}{2},\nu\right)\}d\eta d\nu.
\\&= \int\limits_{\R^{d-2}}\int\limits_{-\frac{|\xi|}{2}}^\infty  \omega(|\xi|)(K(\eta, \nu)+K(|\xi|+\eta,\nu))  - \omega(|\xi|+2\eta)(K(\eta, \nu) - K(|\xi|+\eta, \nu))d\eta d\nu
\\& \quad+\int\limits_{\R^{d-2}}\int\limits_{-\frac{|\xi|}{2}}^\infty  \left[\omega(|\xi|+2\eta) - \overline{u}\left(\frac{|\xi|}{2}+\eta,\nu\right) +\underline{u}\left(-\frac{|\xi|}{2}-\eta,\nu\right)\right] (K(\eta,\nu)-K(|\xi|+\eta,\nu))d\eta d\nu
\\& \geq  \int\limits_{\R^{d-2}}\int\limits_{-\frac{|\xi|}{2}}^\infty  \omega(|\xi|)[K(\eta, \nu)+K(|\xi|+\eta,\nu)]  - \omega(|\xi|+2\eta)[K(\eta, \nu) - K(|\xi|+\eta, \nu)]d\eta d\nu
\end{split}
\end{equation}

It then follows identically to the argument in the appendix of \cite{KiselevIntRearrange} that 
\begin{equation}
\begin{split}
\int\limits_{\R^{d-2}}&\int\limits_{-\frac{|\xi|}{2}}^\infty  \omega(|\xi|)[K(\eta, \nu)+K(|\xi|+\eta,\nu)]  - \omega(|\xi|+2\eta)[K(\eta, \nu) - K(|\xi|+\eta, \nu)]d\eta d\nu
\\ &= \int\limits_0^{\frac{|\xi|}{2}} \left(2\omega(|\xi|) - \omega(|\xi|+2\eta)-\omega(|\xi|-2\eta)\right)\tilde{K}(\eta)d\eta 
 + \int\limits_{\frac{|\xi|}{2}}^\infty \left(2\omega(|\xi|)+ \omega(2\eta-|\xi|) - \omega(2\eta+|\xi|)\right)\tilde{K}(\eta)d\eta,
\end{split}
\end{equation}
where 
\begin{equation}
\tilde{K}(\eta) = \int\limits_{\R^{d-2}} K(\eta, \nu)d\nu = \frac{1}{(\eta^2+\omega(0)^2)^{(2+s)/2}}\left(\int\limits_0^\infty \frac{r^{d-3}}{(1+r^2)^{(d+s)/2}}dr\right).
\end{equation}
Note that 
\begin{equation}
\int\limits_0^\infty \frac{r^{d-3}}{(1+r^2)^{(d+s)/2}}dr = \frac{1}{2}B\left(\frac{d-2}{2}, \frac{2+s}{2}\right) = \frac{\Gamma\left(\frac{d-2}{2}\right)\Gamma\left(\frac{2+s}{2}\right)}{2\Gamma\left(\frac{d+s}{2}\right)},
\end{equation}
where $B(x,y)$ is the Beta function.  

Thus 
\begin{equation}
\begin{split}
\int\limits_{\R^{d-1}} &\frac{\omega(|\xi|)-\left(\overline{u}\left(\frac{\xi}{2}+z\right)-\underline{u}\left(\frac{-\xi}{2}+z\right)\right)}{(|z|^2+\omega(0)^2)^{(d+s)/2}}dz
+\int\limits_{\R^{d-1}}\frac{\overline{u}(z)-\underline{u}(z)}{\max\limits_{\pm}(|z\pm \frac{\xi}{2}|^2+\omega(0)^2)^{(d+s)/2}}dz
\\ &\geq C(d,s)\left( \int\limits_0^{\frac{|\xi|}{2}} \frac{2\omega(|\xi|)-\omega(|\xi|+2\eta)-\omega(|\xi|-2\eta)}{(\eta^2+\omega(0)^2)^{(2+s)/2}}d\eta 
+ \int\limits_{\frac{|\xi|}{2}}^\infty \frac{2\omega(|\xi|)+\omega(2\eta -|\xi|)-\omega(2\eta+|\xi|)}{(\eta^2+\omega(0)^2)^{(2+s)/2}}d\eta\right).
\end{split}
\end{equation}

\end{proof}


\section{Construction of modulus and completion of break through argument}

Combining \eqref{e:preciseboundarydiff} with Lemmas \ref{l:P.V.estimate} and \ref{l:intrearrange}, we have under the assumptions of \eqref{e:omegabound} that 
\begin{equation}\label{e:timeomegabound}
\begin{split}
\partial_t \overline{u}\left(t_0,\frac{\xi}{2}\right)-\partial_t \underline{u}\left(t_0,\frac{-\xi}{2}\right) \lesssim_{d,s,L} &\int\limits_0^{\frac{|\xi|}{2}} \frac{\omega(|\xi|+2\eta)+\omega(|\xi|-2\eta)-2\omega(|\xi|)}{(\eta^2+\omega(0)^2)^{(2+s)/2}}d\eta 
\\&+ \int\limits_{\frac{|\xi|}{2}}^\infty \frac{\omega(2\eta+|\xi|)-\omega(2\eta -|\xi|)-2\omega(|\xi|)}{(\eta^2+\omega(0)^2)^{(2+s)/2}}d\eta.
\end{split}
\end{equation}

As the right hand side of \eqref{e:timeomegabound} only depends on $\omega$, we can now make our choice of a family of moduli of continuity $\omega:[0,\infty)\times[0,\infty)\to [0,\infty)$ to complete the breakthrough argument.  With that in mind, define
\begin{equation}\label{e:omegadefn}
\omega(t,r) = \delta(t)+\left\{\begin{array}{ll} 1+Lr, & r\geq 2 \\(1+L) r - \frac{r^{1+s}}{2^{1+s}},  & c\delta(t)^2\leq r \leq 2, \\ A(\delta)r^2+B(\delta), & 0\leq r\leq c\delta(t)^2 \end{array}\right.
\end{equation}
where $0< c<<1$, and 
\begin{equation}
\begin{split}
A(\delta) &= \frac{1+L}{2c\delta^2} - \frac{1+s}{2^{2+s}} (c\delta^2)^{s-1}, 
\\ B(\delta) &=  \frac{(1+L)c\delta^2}{2} - \frac{1-s}{2^{2+s}}(c\delta^2)^{1+s}.
\end{split}
\end{equation}
 are chosen so that $\omega(t, \cdot)$ is $C^1$, and $\delta:[ 0,T]\to [0,1]$ is a to be determined non increasing function with $\delta(0)= 1$ and $\delta(T) = 0$.  The function $\delta(t)\approx \omega(t,0)$ and essentially represents how far the boundary of our set $E_t$ is from being a graph.

\begin{lemma}\label{l:omegatimederivbound}
For $\omega(t,r)$ as defined in \eqref{e:omegadefn}, $\partial_t \omega(t,r)> 2\delta'(t)$ for $0< c< \frac{2}{5(1+L)}$.  
\end{lemma}
\begin{proof}
Examining the formula in \eqref{e:omegadefn} and the fact that $\delta'(t)\leq 0$, its clear that it suffices to show
\begin{equation}
A'(\delta(t))r^2 + B'(\delta(t))< 1, \qquad 0\leq r\leq c\delta(t)^2, \ 0\leq c\leq\frac{2}{5(1+L)}.
\end{equation}

Differentiating $B$ and using that $0\leq \delta\leq 1$, we have that 
\begin{equation}
B'(\delta)  =(1+L)c\delta -\frac{(1-s)(1+s)}{2^{1+s}}c^{1+s}\delta^{1+2s} \leq (1+L)c< \frac{2}{5}.
\end{equation}
Similarly, as 
\begin{equation}
A'(\delta) = \frac{-3(1+L)}{2c\delta^3} + \frac{(1+s)(3-2s)}{2^{2+s}c^{1-s}\delta^{3-2s}}
\end{equation}
and $0\leq r\leq c\delta^2$, we have that 
\begin{equation}
A'(\delta)r^2 \leq \frac{(1+s)(3-2s)}{2^{2+s}c^{1-s}\delta^{3-2s}}r^2\leq \frac{3}{2}c^{1+s}\delta^{1+2s}\leq \frac{3}{2}c< \frac{3}{5}.
\end{equation}
\end{proof}

Thus $\partial_t \omega(t,\cdot)$ will always be comparable to $\delta'(t)$.  Our goal now is to bound 
\begin{equation}\label{e:omegaintegrals}
 \int\limits_0^{\frac{|\xi|}{2}} \frac{\omega(t,|\xi|+2\eta)+\omega(t,|\xi|-2\eta)-2\omega(t,|\xi|)}{(\eta^2+\omega(t,0)^2)^{(2+s)/2}}d\eta + \int\limits_{\frac{|\xi|}{2}}^\infty \frac{\omega(t,2\eta+|\xi|)-\omega(t,2\eta -|\xi|)-2\omega(t,|\xi|)}{(\eta^2+\omega(t,0)^2)^{(2+s)/2}}d\eta,
\end{equation}
in terms of $\delta(t)\approx \omega(t,0)$.  

If $\omega(t,\cdot)$ was a concave function, then both of the integrals in \eqref{e:omegaintegrals} would be nonpositive.  However because we needed $\partial_r \omega(t,0) = 0$ in our construction of $\omega$ in case the touching point $\xi = 0$, $\omega(t,\cdot)$ will be convex near 0.  Thus the first integral of \eqref{e:omegaintegrals} can be positive.  However, we will show that as long as $c$ is taken small, it will be under control.

\begin{lemma}\label{l:omegaposintbound}
Let $\omega(t,r)$ be as defined in \eqref{e:omegadefn} and $|\xi|\leq 2$.  Then 
\begin{equation}\label{e:omegaposint}
 \int\limits_0^{\frac{|\xi|}{2}} \frac{\omega(t,|\xi|+2\eta)+\omega(t,|\xi|-2\eta)-2\omega(t,|\xi|)}{(\eta^2+\omega(t,0)^2)^{(2+s)/2}}d\eta\leq(1+L)c^2.
\end{equation}
\end{lemma}

\begin{proof}

We claim that 
\begin{equation}\label{e:omegaposupperbound}
\omega(t,|\xi|+2\eta)+\omega(t,|\xi|-2\eta)- 2\omega(t,|\xi|) \leq \left\{\begin{array}{ll} 0, &  2\eta - |\xi|\geq c\delta^2, \\ (1+L)c\delta^2, & \text{ otherwise}  \end{array}\right.
\end{equation}

Given \eqref{e:omegaposupperbound}, it follows immediately that for $|\xi|\leq c\delta^2$
\begin{equation}
\begin{split}
\int\limits_0^{\frac{|\xi|}{2}} \frac{ \omega(t,|\xi|+2\eta)+\omega(t,|\xi|-2\eta)- 2\omega(t,|\xi|)}{(\eta^2+\omega(t,0)^2)^{(2+s)/2}}d\eta & \leq\int\limits_{0}^{c\delta(t)^2} \frac{(1+L)c\delta(t)^2}{\delta(t)^{2+s}} d\eta 
\\& =(1+L)c^2\delta(t)^{2-s}.
\end{split}
\end{equation}
For $|\xi|\geq c\delta(t)^2$, we similarly have that 
\begin{equation}
\begin{split}
\int\limits_0^{\frac{|\xi|}{2}} \frac{ \omega(t,|\xi|+2\eta)+\omega(t,|\xi|-2\eta)- 2\omega(t,|\xi|)}{(\eta^2+\omega(t,0)^2)^{(2+s)/2}}d\eta & \leq\int\limits_{\frac{|\xi| - c\delta(t)^2}{2}}^{\frac{|\xi|}{2}} \frac{(1+L)c\delta(t)^2}{\delta(t)^{2+s}} d\eta 
\\& = (1+L)c^2\delta(t)^{2-s}.
\end{split}
\end{equation}
As $\delta \leq 1$ always, we thus have that 
\begin{equation}
\begin{split}
\int\limits_0^{\frac{|\xi|}{2}} \frac{ \omega(t,|\xi|+2\eta)+\omega(t,|\xi|-2\eta)- 2\omega(t,|\xi|)}{(\eta^2+\omega(t,0)^2)^{(2+s)/2}}d\eta & \leq(1+L)c^2,
\end{split}
\end{equation}
for all $|\xi|\leq 2$.  

All that remains is to prove the claim \eqref{e:omegaposupperbound}.  Consider the function 
\begin{equation}\label{e:tildeomega}
\tilde{\omega}(t,r) =  \delta(t)+\left\{\begin{array}{ll} 1+Lr, & r\geq 2 \\ (1+L)r - \frac{r^{1+s}}{2^{1+s}},  & 0\leq r \leq 2,\end{array}\right. .
\end{equation}
Then for fixed $t$, $\tilde{\omega}(t,\cdot)$ is a concave function of $r$.  Thus whenever $|\xi|-2\eta\geq c\delta(t)^2$, we have that 
\begin{equation}
\omega(t,|\xi|+2\eta)+\omega(t,|\xi|-2\eta)- 2\omega(t,|\xi|) = \tilde{\omega}(t,|\xi|+2\eta)+\tilde{\omega}(t,|\xi|-2\eta)- 2\tilde{\omega}(t,|\xi|)\leq 0.
\end{equation} 
As we also have that 
\begin{equation}
0\leq \omega(t,r)-\tilde{\omega}(t,r)\leq \omega(t,0)-\tilde{\omega}(t,0) = B(\delta(t)) \leq \frac{(1+L)c\delta(t)^2}{2},
\end{equation}
it follows that when $|\xi|-2\eta\leq c\delta(t)^2$
\begin{equation}
\begin{split}
\omega(t,|\xi|+2\eta)+\omega(t,|\xi|-2\eta)- 2\omega(t,|\xi|) &\leq \tilde{\omega}(t,|\xi|+2\eta)+\tilde{\omega}(t,|\xi|-2\eta)- 2\tilde{\omega}(t,|\xi|)+(1+L)c\delta(t)^2 
\\&\leq (1+L)c\delta(t)^2.
\end{split}
\end{equation}

\end{proof}

With Lemma \ref{l:omegaposintbound}, we can bound the first integral in \eqref{e:omegaintegrals} by an arbitrarily small constant as $c\to 0$.  All that remains now is to get a good, negative upper bound on the second integral.

\begin{lemma}\label{l:omeganegintbound}
Let $\omega(t,r)$ be as defined in \eqref{e:omegadefn} and $|\xi|\leq 2$.  Then 
\begin{equation}\label{e:omeganegint}
\int\limits_{\frac{|\xi|}{2}}^\infty \frac{\omega(t,2\eta+|\xi|)-\omega(t,2\eta -|\xi|)-2\omega(t,|\xi|)}{(\eta^2+\omega(t,0)^2)^{(2+s)/2}}d\eta\lesssim -1.
\end{equation}
\end{lemma}

\begin{proof}
Again, take  $\tilde{\omega}(t,r)$ to be  
\begin{equation}
\tilde{\omega}(t, r) = \delta(t)+\left\{\begin{array}{ll} 1+Lr, & r\geq 2 \\ (1+L)r - \frac{r^{1+s}}{2^{1+s}},  & 0 \leq r \leq 2, \end{array}\right. 
\end{equation}
Then we claim that for $\eta\geq \displaystyle\frac{|\xi|}{2}$,
\begin{equation}\label{e:deltaboundbelow}
\omega(t,2\eta+|\xi|)- \omega(t,2\eta-|\xi|) - 2\omega(t,|\xi|) \leq \left\{\begin{array}{ll} \tilde{\omega}(t,2\eta+|\xi|)- \tilde{\omega}(t,2\eta-|\xi|) - 2\tilde{\omega}(t,|\xi|), & |\xi| \geq \frac{\delta(t)}{2(1+L)},, \\ -\delta(t), & |\xi|\leq \frac{\delta(t)}{2(1+L)},  \end{array} \right.
\end{equation}
To see this, note that $\tilde{\omega}(t,r)\leq \omega(t,r)$ with equality for $r\geq c\delta(t)^2$ by the definition of $\tilde{\omega}$ \eqref{e:tildeomega}.  Thus in the case that $\displaystyle |\xi| \geq \frac{\delta(t)}{2(1+L)}, \eta\geq \frac{|\xi|}{2}$, we have that $2\eta+|\xi|\geq \displaystyle\frac{\delta(t)}{1+L}\geq c\delta(t)^2$  so \eqref{e:deltaboundbelow} follows immediately.  And when $|\xi|\leq \displaystyle\frac{\delta(t)}{2(1+L)}$, we then have that  
\begin{equation}
\omega(t,2\eta+|\xi|)- \omega(t,2\eta-|\xi|) - 2\omega(t,|\xi|) \leq 2|\xi|(1+L)- 2\omega(t,0)\leq -\delta(t).
\end{equation}
Thus we've proven \eqref{e:deltaboundbelow}.  Note that as 
\begin{equation}
\tilde{\omega}(t,2\eta+|\xi|)- \tilde{\omega}(t,2\eta-|\xi|) - 2\tilde{\omega}(t,|\xi|) \leq \tilde{\omega}(t,2\eta)- \tilde{\omega}(t,2\eta) - 2\tilde{\omega}(t,0)\leq -2\delta(t),
\end{equation}
we always have that 
\begin{equation}\label{e:deltaboundbelow2}
\omega(t,2\eta+|\xi|)- \omega(t,2\eta-|\xi|) - 2\omega(t,|\xi|) \leq -\delta(t).
\end{equation}

Now to prove \eqref{e:omeganegint}, we will consider three cases.  First consider small $\xi$, where $|\xi|\leq \delta(t)$.  
Then using \eqref{e:deltaboundbelow2} we get that 
\begin{equation}
\begin{split}
\int\limits_{\frac{|\xi|}{2}}^\infty \frac{\omega(t,2\eta+|\xi|)- \omega(t, 2\eta-|\xi|) - 2\omega(t, |\xi|)}{ (\eta^2+\omega(t,0)^2)^{(2+s)/2}}d\eta &\lesssim \int\limits_{\frac{\delta(t)}{2}}^{\delta(t)} \frac{-\delta(t)}{\delta(t)^{2+s}}d\eta 
\\&\lesssim -\delta(t)^{-s}\leq -1.  
\end{split}
\end{equation}
as $\delta \leq 1$.  

For the second case, we consider midsize $\xi$ where $\delta(t)\leq |\xi|\leq \displaystyle\frac{1}{2}$.  Then using \eqref{e:deltaboundbelow}, we have that 
\begin{equation}
\begin{split}
\int\limits_{\frac{|\xi|}{2}}^\infty \frac{\omega(t,2\eta+|\xi|)- \omega(t,2\eta-|\xi|) - 2\omega(t,|\xi|)}{ (\eta^2+\omega(t,0)^2)^{(2+s)/2}}d\eta &\lesssim\int\limits_{\frac{|\xi|}{2}}^{\frac{1}{2}} \frac{-(2\eta+|\xi|)^{1+s} + (2\eta - |\xi|)^{1+s} + 2|\xi|^{1+s}}{\eta^{2+s}}d\eta 
\\&= \int\limits_{\frac{1}{2}}^{\frac{1}{2|\xi|}} \frac{-(2\tilde{\eta} + 1)^{1+s} + (2\tilde{\eta}-1)^{1+s} + 2}{\tilde{\eta}^{2+s}}d\tilde{\eta} 
\\&\leq \int\limits_{\frac{1}{2}}^{1} \frac{-(2\tilde{\eta} + 1)^{1+s} + (2\tilde{\eta}-1)^{1+s} + 2}{\tilde{\eta}^{2+s}}d\tilde{\eta}  
\\& \leq\int\limits_{\frac{1}{2}}^{1} \frac{-2^{1+s}+ 2}{\tilde{\eta}^{2+s}}d\tilde{\eta}  
 \leq \int\limits_{\frac{1}{2}}^1 \frac{-2\ln(2)s}{\tilde{\eta}^{2+s}}d\tilde{\eta}
 \\&\lesssim -s.  
\end{split}
\end{equation}

Finally, suppose that $\xi$ is large so $2\geq |\xi|\geq \displaystyle\frac{1}{2}$.  Then $\omega(t,|\xi|)\geq \displaystyle\frac{1}{4}+L|\xi|$, so 
\begin{equation}
\begin{split}
\int\limits_{\frac{|\xi|}{2}}^\infty \frac{\omega(t,2\eta+|\xi|)- \omega(t,2\eta-|\xi|) - 2\omega(t,|\xi|)}{ (\eta^2+\omega(t,0)^2)^{(2+s)/2}}d\eta &\lesssim \int\limits_{2}^{\infty} \frac{-2\omega(t,|\xi|)}{\eta^{2+s}}d\eta 
\\&\lesssim -1
\end{split}
\end{equation}
\end{proof}

Take $c<<\displaystyle\frac{s}{1+L}$ and $\delta(t) = \displaystyle\frac{T-t}{T}$ in \eqref{e:omegadefn} for some
\begin{equation} \label{e:Tdefn}
T\gtrsim \frac{(3(1+L))^{d+s}\Gamma\left(\frac{d+s}{2}\right)}{s^2(1-s)\Gamma\left(\frac{d-2}{2}\right)\Gamma\left(\frac{2+s}{2}\right)}, \quad d\geq 3,  \qquad \left( T\gtrsim  \frac{(1+L)^{2+s}}{s^2(1-s)}, \quad d=2\right)
\end{equation}
Then combining Lemmas \ref{l:P.V.estimate} through \ref{l:omeganegintbound}, we have under the assumptions of the break through argument at the end of Section 2 that 
\begin{equation}
\begin{split}
\partial_t \overline{u}&\left(t_0,\frac{\xi}{2}\right) - \partial_t \underline{u}\left(t_0,\frac{-\xi}{2}\right) 
\\&= s(1-s)\sqrt{1+ \partial_r\omega(t,|\xi|)^2}  \int\limits_{\R^{d-1}}\int\limits_\R \frac{\1_E^\pm\left(\frac{\xi}{2}+z, \overline{u}\left(t_0,\frac{\xi}{2}\right)+z_d\right)-\1_E^\pm\left(\frac{-\xi}{2}+z, \underline{u}\left(t_0,\frac{-\xi}{2}\right)+z_d\right)}{(|z|^2+z_d^2)^{(d+s)/2}}dzdz_d
\\&\leq C(d,s,L)\left( \int\limits_0^{\frac{|\xi|}{2}} \frac{\omega(|\xi|+2\eta)+\omega(|\xi|-2\eta)-2\omega(|\xi|)}{(\eta^2+\omega(0)^2)^{(2+s)/2}}d\eta + \int\limits_{\frac{|\xi|}{2}}^\infty \frac{\omega(2\eta+|\xi|)-\omega(2\eta -|\xi|)-2\omega(|\xi|)}{(\eta^2+\omega(0)^2)^{(2+s)/2}}d\eta\right)
\\&\leq \frac{-2}{T} = 2\delta'(t_0) < \partial_t \omega(t, |\xi|),
\end{split}
\end{equation}
a contradiction.  Thus for any smooth flow $t\to E_t$ with with initial data $E_0$ satisfying the modulus $1+Lr$ and flat at infinity, we have that $E_t$ has modulus $\omega(t,\cdot)$ for all $t\in [0,T]$.  In particular, $\partial E_T$ is a $(1+L)$-Lipschitz graph.


\section{Viscosity Solutions and Technicalities}

Sections 3 through 5 gave the proof of Theorem \ref{t:main} in the case that we have a smooth flow.  But even for smooth initial data, there's no guarantee a unique, smooth solution of fractional mean curvature flow exists.  So instead we work with viscosity solutions.  See the appendix or \cite{CyrilFract,Chambolle} for appropriate definitions and details.  

Fix a tuple $(E_0^-, \Gamma_0, E_0^+)$ with $E_0^\pm\subseteq\R^d$ open, $\Gamma_0\subseteq\R^d$ closed, all are disjoint and $E_0^-\cup \Gamma_0 \cup E_0^+ = \R^d$.  We then have that there is a unique viscosity solutions $(E_t^-, \Gamma_t, E_t^+)$ of \eqref{e:meancurvfloweqnset} in the sense of Definition \ref{d:viscosityset} for all times $t$.  

If we knew a priori that $\mathcal{L}^d(\Gamma_t) = 0$ for all times $t$, then we could repeat the same argument as in the smooth case with only minor alterations.  
However that is not the case in general, so we must adjust.  

Our first goal is to prove Theorem \ref{t:main} under the assumptions that 
\begin{equation}\label{e:initialassumptions}
\left\{\begin{array}{l} 1). E_0^\pm \text{ have modulus } (1-\eta)+Lr,
\\ 2). 0 \in \Gamma_0
\\ 3). \Gamma_0 \setminus (B_M^{d-1}\times \R) = \{(x,0): x\in \R^{d-1}, |x|\geq M\},
\end{array}\right.
\end{equation}
for some $0<\eta<<1$ and $1<<M<\infty$.  We will later be able to remove the last condition and let $\eta\to 0$.  But for now these are convenient assumptions.  

Let $U_0: \R^d \to \R$ be the signed distance function 
\begin{equation}\label{e:U0}
U_0(X) = \left\{\begin{array}{ll} \min\{d(X, \Gamma_0),1\}, & X\in E_0^+, \\ \max\{-d(X,\Gamma_0), -1\}, & X\in E_0^-  \end{array}\right. ,
\end{equation}
and let $U(t,X)$ be the unique viscosity solution to the level set equation \eqref{e:viscosity} for the initial data $U_0$.  Then for any $\gamma\in [-1,1]$, we can define the tuple $(E_t^{\gamma-},\Gamma_t^\gamma, E_t^{\gamma+})$ by 
\begin{equation}\label{e:gammatupledefn}
E_t^{\gamma -} = \{U(t,\cdot)<\gamma\},\quad \Gamma_t^\gamma = \{U(t,\cdot) = \gamma\}, \quad E_t^{\gamma+} = \{U(t,\cdot)>\gamma\}.
\end{equation}
Note that $(E_t^{0-}, \Gamma_t^0, E_t^{0+}) = (E_t^-, \Gamma_t, E_t^+)$ is our original viscosity solution triple.  

\begin{lemma}\label{l:gammainitialmod}
Let $(E_0^-,\Gamma_0, E_0^+)$ satisfy \eqref{e:initialassumptions} and $(E_t^{\gamma -}, \Gamma_t^\gamma, E_t^{\gamma +} )$ be as in \eqref{e:gammatupledefn}.  Then $E_t^{\gamma\pm} $ have modulus of continuity $1+Lr$ for all times $t$ and $|\gamma|<\displaystyle\frac{\eta}{\sqrt{1+L^2}}$.  
\end{lemma}
\begin{proof}
Let $A = \{(x,x_d):  x_d\leq 1-\eta + L|x|\}$ and $B = \{(y,y_d):  y_d\geq 1+L|y|\}$.  Direct calculation then gives that 
\begin{equation}
d(A,B ) = \displaystyle\frac{\eta}{\sqrt{1+L^2}}.
\end{equation}
Using the set formulation of a modulus of continuity (Definition \ref{d:setmodulusdefn}), checking cases then gives you that $E_0^{\gamma\pm}$ have modulus $1+Lr$ for $|\gamma|\displaystyle\leq \frac{\eta}{\sqrt{1+L^2}}$.  Proposition \ref{p:propagation} then implies that this remains true for all $t\geq 0$.  
\end{proof}

\begin{proposition}
Let $(E_t^{\gamma-},\Gamma_t^\gamma, E_t^{\gamma+})$ be as in \eqref{e:gammatupledefn}.  Then for almost every $\gamma\in [-1,1]$, 
\begin{equation}\label{e:0boundaryaet}
\mathcal{L}^{d}(\Gamma_t^\gamma) = 0 \text{ for almost every time } t\geq 0.
\end{equation}
\end{proposition}
\begin{proof}
This follows easily from the fact that there are at most countably many $\gamma\in \R$ such that 
\begin{equation}
\mathcal{L}^{d+1}(\{U(\cdot, \cdot)=\gamma\}) = \mathcal{L}^{d+1}\left(\bigcup\limits_{t\geq 0} \{t\}\times \Gamma_t^\gamma\right)\not= 0
\end{equation}
\end{proof}

As a slight abuse of notation, we define $\overline{u}^\gamma,\underline{u}^\gamma: [0,\infty)\times \R^{d-1}\to \R$ by
\begin{equation}
\overline{u}^\gamma(t,x) = \max\{x_d| (x,x_d)\in \Gamma_t^\gamma\}, \quad \underline{u}^\gamma(t,x) = \min\{x_d| (x,x_d)\in \Gamma_t^\gamma\}.
\end{equation}

Our goal now is to show that for any $\gamma$ such that Lemma \ref{l:gammainitialmod} and \eqref{e:0boundaryaet} hold, $\Gamma_t^\gamma$ becomes a $(1+L)$-Lipschitz graph in finite time.  Explicitly, 

\begin{lemma}\label{l:assumptionscase}
Assume $(E_0^-, \Gamma_0, E_0^+)$ satisfy \eqref{e:initialassumptions}.  Then for any $|\gamma|<\displaystyle\frac{\eta}{\sqrt{1+L^2}}$  with 
\begin{equation}
\mathcal{L}^{d}(\Gamma_t^\gamma) = 0 \text{ for almost every time } t\geq 0,
\end{equation}
$\Gamma_t^\gamma$ has modulus $\omega\left(\frac{t}{2}, \cdot\right)$ for times $t\in [0,2T]$.  That is, 
\begin{equation}
\overline{u}^\gamma(t,x)-\underline{u}^\gamma(t,y)\leq \omega\left(\frac{t}{2}, |x-y|\right),
\end{equation}
for all $(t,x,y)\in [0,2T]\times \R^{d-1}\times \R^{d-1}$ where $\omega$ is as in \eqref{e:omegadefn} and $T$ is as in \eqref{e:Tdefn}.
\end{lemma}

\begin{remark}
For simplicity we prove that $E_t^{\gamma\pm}$ have modulus $\omega(\frac{t}{2}, \cdot)$, but similar arguments can be made to show that $E_t^{\gamma\pm}$ has modulus $\omega((1-\epsilon)t, \cdot)$ for any $\epsilon>0$, and hence modulus $\omega(t, \cdot)$ by continuity. 
\end{remark}

Note that by continuity, it suffices to prove that $\Gamma_t^\gamma$ has modulus $\omega(\frac{t}{2},\cdot)$ at times $t\leq t_0$, for some arbitrary $t_0< 2T$.  The key advantage to proving this is that 
\begin{equation}
\inf\left\{\omega\left(\frac{t}{2},r\right): r\geq 0, 0\leq t\leq t_0\right\} = \omega\left(\frac{t_0}{2}, 0\right)\geq \delta\left(\frac{t_0}{2}\right)>0.
\end{equation}

The $\omega$ bound from below and Proposition \ref{p:limitlemma} then allows us to rule out any crossing points at infinity.  

\begin{proposition}\label{p:limitlemma}
Suppose that our initial data $U_0: \R^d\to [-1,1]$ is as in \eqref{e:U0} for some tuple $(E_0^-, \Gamma_0, E_0^+)$ satisfying \eqref{e:initialassumptions}
Then for any $\delta>0$ and $t_0<\infty$, there exists an $r=r(M,\delta, t_0)<\infty$ such that 
\begin{equation}
|U(t,x,x_d)-U_0(x,x_d)|<\delta/2, \qquad |x|\geq r(M,\delta,t_0), \ t\in [0,t_0].
\end{equation}
\end{proposition}

\begin{proof}

Let $\phi\in C^{\infty}_c(\R)$ with $\phi\geq 0$, $\phi\equiv 1$ on $[-1/2,1/2]$, and $\text{supp}(\phi)\subseteq [-1,1]$.  Let 
\begin{equation}
V_r(t,x,x_d) = \frac{\delta }{2t_0}t+2\phi\left(\frac{|x|}{r}\right) + \min\{(x_d+1)_+-1 ,1\}
\end{equation}
Then for $r\geq 2M$, we have that $V_r(0,x,z) \geq U_0(x,z)$.  Since $\phi\in C^\infty_c(\R)$, we have that if we take $r$ sufficiently large depending on $\displaystyle\frac{\delta}{t_0}$, then $V_r$ will be a supersolution to \eqref{e:viscositysuper}.  Thus by the comparison principle Theorem \ref{t:comparison}, for any $t\in [0,t_0]$ and $|x|\geq r$ and $x_d\in\R$
\begin{equation}
U(t,x,x_d)\leq V_r(t,x,x_d) = \frac{\delta }{2t_0}t +\min\{(x_d+1)_+-1 ,1\}  \leq  \frac{\delta}{2}+\min\{(x_d+1)_+-1 ,1\} = \frac{\delta}{2}+U_0(x,x_d) .
\end{equation}
A symmetric proof works for the opposite inequality.  
\end{proof}

The fact that we are restricting ourselves to times $0\leq t\leq t_0 < 2T$ and Proposition \ref{p:limitlemma} effectively allows us to redo the breakthrough argument of section 2.  However, when we try to redo the estimates from section 3, we run into a problem because our ``boundary" $\Gamma_t^\gamma$ might have positive measure.  

Again, if we knew in fact that $\mathcal{L}^d(\Gamma_t^\gamma) = 0$ for all times $t$, then the argument from the smooth case would work with minor alterations.  The main problem when we only have $\mathcal{L}^d(\Gamma_t^\gamma) = 0$ for a.e. time $t$ is that at the breakthrough argument relies on having an open interval of times where we can run it.  Else at the crossing time $t_0$ we have no guarantee that $\mathcal{L}^d(\Gamma_{t_0}^\gamma) = 0$.  

Thus in order to deal with this, we're going need to adjust the modulus estimates in Section 3 to work when we only have that $\mathcal{L}^d(\Gamma_t^\gamma)$ is \emph{small}.  Luckily, that will be true on an open interval of times.  

\begin{proposition}\label{p:semicontinuity}
Let $\gamma\in \R$ and $R>0$.  Then the function 
\begin{equation}
t\to \mathcal{L}^d (\Gamma_t^\gamma\cap B_R^d) 
\end{equation}
is upper semicontinuous.  In particular, for any $\epsilon>0$ the set of times 
\begin{equation}
\{t\in (0,T):  \mathcal{L}^d (\Gamma_t^\gamma\cap B_R^d) <\epsilon\}
\end{equation}
is open.  
\end{proposition}
\begin{proof}
Let $\epsilon>0$ and $t\in [0,\infty)$.  Then it suffices to show that there is some $\delta>0$ such that if $|t-t'|<\delta$ then 
\begin{equation}
 \mathcal{L}^d (\Gamma_{t'}^\gamma\cap B_R^d) <  \mathcal{L}^d (\Gamma_t^\gamma\cap B_R^d) +\epsilon.
\end{equation}
Let $N_r(\Gamma_t^\gamma) = \{X\in\R^d: d(X,\Gamma_t^\gamma)< r\}.$  Then for $r$ sufficiently small, 
\begin{equation}
 \mathcal{L}^d (N_r(\Gamma_{t}^\gamma)\cap B_R^d) <  \mathcal{L}^d (\Gamma_t^\gamma\cap B_R^d) +\epsilon.
\end{equation}
Let $K = \{t\}\times (\overline{B_R^d}\setminus N_r(\Gamma_t^\gamma))$.  Then $K$ is compact with $K\cap U^{-1}(\gamma) = \emptyset$.  Hence, $d(K, U^{-1}(\gamma)) = \delta>0$ for some $\delta$.  In particular, $|t-t'|<\delta$ implies that 
$ (\overline{B_R^d}\setminus N_r(\Gamma_t^\gamma)) \cap \Gamma_{t'}^\gamma = \emptyset$.  Thus 
\begin{equation}
 \mathcal{L}^d (\Gamma_{t'}^\gamma\cap B_R^d)  \leq  \mathcal{L}^d (N_r(\Gamma_{t}^\gamma)\cap B_R^d)<  \mathcal{L}^d (\Gamma_t^\gamma\cap B_R^d) +\epsilon.
\end{equation}
\end{proof}

Now with Proposition \ref{p:semicontinuity}, we first make our choice of $R= R(t_0)$ as 
\begin{equation}
R(t_0) = r\left(M,\delta\left(\frac{t_0}{2}\right), t_0\right)+3+LM+ \left(\frac{1}{8(2+L)(1-s)T}\right)^{-1/s}.
\end{equation}
where $r\left(\displaystyle M,\delta\left(\frac{t_0}{2}\right), t_0\right)$ is as in Proposition \ref{p:limitlemma}.  
With this choice of $R(t_0)$, by Proposition \ref{p:semicontinuity}  we have the set of times 
\begin{equation}
\mathcal{T}(t_0) = \left\{t\in (0,t_0):  \mathcal{L}^d (\Gamma_t^\gamma\cap B_{R(t_0)}^d) < \frac{\epsilon(t_0)^{d+s}}{8(2+L)s(1-s)T}\right\} = \bigcup\limits_{i}(a_i,b_i),
\end{equation}
is an open set of full measure on $(0,t_0)$, where 
\begin{equation}
\epsilon(t_0) = \left(\frac{c\delta(\frac{t_0}{2})^2}{8(2+L)sT}\right)^{\frac{1}{(1-s)}}.
\end{equation}

The breakthrough argument of section 2 is designed to work on an open interval of times.  However, a finite union of open intervals works just as well.  For $N\in \N$, define $\alpha_{N}:[0,t_0]\to [0,t_0]$ by $\alpha_{N}(0)=0$ and 
\begin{equation}
\alpha_{N}'(t) = \left\{\begin{array}{ll} \frac{1}{2}, & t\in (a_i, b_i) \ \text{ for } i=1,\ldots ,N,  \\ 0, & \text{otherwise} \end{array}\right. .
\end{equation}
Note that $\displaystyle\lim\limits_{N\to \infty} \alpha_N(t) = \frac{t}{2}$ for $t\in [0,t_0]$ as $\mathcal{T}(t_0)$ has full measure.  Thus it would suffice to show that $\Gamma_t^\gamma$ has modulus $\omega(\alpha_N(t),\cdot)$ for every $t\in [0,t_0]$ to prove Lemma \ref{l:assumptionscase}.

With all of this set up, we can now define our test function $$\Phi: [0,t_0]\times \R^{d-1}\times \R\times \R^{d-1}\times \R\to \R$$ by
\begin{equation}\label{e:phidefn}
\Phi(t,x,x_d, y, y_d) = \left(\omega(\alpha_N(t), |x-y|) -(x_d-y_d)+\frac{1}{2}\right)_+ +\frac{1}{2} .
\end{equation}
The function $\Phi$ encodes that $\Gamma_t^\gamma$ has the modulus $\omega(\alpha_N(t),\cdot)$, in the sense that

\begin{lemma}\label{l:phimodulus}
$\Gamma_t^\gamma$ has modulus $\omega(\alpha_N(t), \cdot)$ if and only if for all $X,Y\in \R^d$,
\begin{equation}
\1_{E_t^{\gamma-}\cup\Gamma_t^\gamma}(X) - \1_{E_t^{\gamma-}}(Y)\leq \Phi(t,X,Y).
\end{equation}
Similarly, $\Gamma_t^\gamma$ only has the modulus strictly if the inequality above is strict.  
\end{lemma}

\begin{proof}
We simply show the first statement, as the second follows similarly.  

One direction is straightforward, as if we take $X = (x,\overline{u}^\gamma(t,x))$ and $Y = (y, \underline{u}^\gamma(t,y))$ then 
\begin{equation}
\begin{array}{c}
1 = \1_{E_t^{\gamma-}\cup\Gamma_t^\gamma}(X) - \1_{E_t^{\gamma-}}(Y)\leq \Phi(t,X,Y) = \omega(\alpha_{N}(t),|x-y|) -(\overline{u}^\gamma(t,x) - \underline{u}^\gamma(t,y))+1,
\\ \Rightarrow \overline{u}^\gamma(t,x)-\underline{u}^\gamma(t,y)\leq \omega(\alpha_N(t),|x-y|).
\end{array}
\end{equation}
As for the converse, suppose that $\overline{u}^\gamma(t,x)-\underline{u}^\gamma(t,y)\leq \omega(\alpha_N(t),|x-y|)$ for all $x,y\in \R^{d-1}$.  
Then since indicator functions can only take the values 0 or 1 and $\Phi>0$, it follows immediately that if $\1_{E_t^{\gamma-}\cup\Gamma_t^\gamma}(X) - \1_{E_t^{\gamma-}}(Y) \not = 1$ then 
\begin{equation}
\1_{E_t^{\gamma-}\cup\Gamma_t^\gamma}(X) - \1_{E_t^{\gamma-}}(Y) \leq 0< \Phi(t,X,Y).
\end{equation}
In the case that $\1_{E_t^{\gamma-}\cup\Gamma_t^\gamma}(X) - \1_{E_t^{\gamma-}}(Y)=1$, we have that $X\in E_t^{\gamma-}\cup\Gamma_t^\gamma$ and $Y\not\in E_t^{\gamma-}$.  Using the monotonicity of $\Phi$ in the $x_d,y_d$ variables we have that
\begin{equation}
\begin{split}
 \1_{E_t^{\gamma-}\cup\Gamma_t^\gamma}(X) - \1_{E_t^{\gamma-}}(Y)=  1&\leq1+ \omega(\alpha_N(t),|x-y|)- (\overline{u}^\gamma(t,x)-\underline{u}^\gamma(t,y))
 \\&\leq  \omega(\alpha_N(t), |x-y|) -(x_d-y_d)+1 = \Phi(t,X,Y).
 \end{split}
\end{equation}
\end{proof}

Thus in order to prove Lemma \ref{l:assumptionscase} it suffices to show that 
\begin{equation}
\1_{E_t^{\gamma-}\cup\Gamma_t^\gamma}(X) - \1_{E_t^{\gamma-}}(Y)\leq \Phi(t,X,Y), \quad \text{ for all } (t,X,Y)\in [0,t_0]\times \R^d\times \R^d.
\end{equation}

With the help of our assumptions \eqref{e:initialassumptions}, Proposition \ref{p:limitlemma} and the definition \eqref{e:omegadefn} of $\omega$, we can now formally justify a large portion of the breakthrough argument by showing that
\begin{lemma}\label{l:phibounds}
Let $\Phi$ be as in \eqref{e:phidefn}.  Then for $|\gamma|<\displaystyle\frac{\eta}{\sqrt{1+L^2}}$, 
\begin{equation}
\left\{\begin{array}{ll} \Phi(0,X,Y)> \1_{E_0^{\gamma-}\cup\Gamma_0^\gamma}(X) - \1_{E_0^{\gamma-}}(Y), & X,Y\in \R^d,
\\\Phi(t,x,x_d,y,y_d)>\1_{E_t^{\gamma-}\cup\Gamma_t^\gamma}(x,x_d) - \1_{E_t^{\gamma-}}(y,y_d), & t\in [0,t_0], x,y\in \R^{d-1}, |x_d| \text{ or }|y_d| >1+LM,
\\ \Phi(t,x,x_d,y,y_d)>\1_{E_t^{\gamma-}\cup\Gamma_t^\gamma}(x,x_d) - \1_{E_t^{\gamma-}}(y,y_d), & t\in [0,t_0], |x| \text{ or } |y| > r(M,\delta(t_0/2), t_0)+2.
\end{array}\right. 
\end{equation}
\end{lemma}

\begin{proof}
When $t=0$, by Lemma \ref{l:gammainitialmod} and the definition \eqref{e:omegadefn} we have that $E_t^{\gamma\pm}$ strictly has the modulus $\omega(0,r)>1+Lr$.  Thus by Lemma \ref{l:phimodulus}
\begin{equation}
\Phi(0,X,Y)>\1_{E_0^{\gamma-}\cup\Gamma_0^\gamma}(X)-\1_{E_t^{\gamma-}}(Y).
\end{equation}

Note that our assumptions \eqref{e:initialassumptions} on $(E_0^-, \Gamma_0, E_0^+)$ imply that 
\begin{equation}
 \{\pm x_d\geq 1-\eta +LM\}\subseteq E_0^{\pm}
\end{equation}
Thus for $|\gamma|<\displaystyle\frac{\eta}{\sqrt{1+L^2}}$, we have that 
\begin{equation}
 \{\pm x_d\geq 1 +LM\}\subseteq E_0^{\gamma\pm}.
\end{equation}
By the comparison principle, this remains true for all later times.  So if $x_d>1+LM$ or $y_d < -1-LM$, then 
\begin{equation}
\Phi(t,x,x_d,y,y_d) > 0 \geq \1_{E_t^{\gamma-}\cup\Gamma_t^\gamma}(x,x_d) - \1_{E_t^{\gamma-}}(y,y_d).
\end{equation}
On the other hand if $x_d<-1-LM$ and $y_d\geq -1-LM$, then $x_d\leq y_d$ and hence
\begin{equation}
\Phi(t,x,x_d,y,y_d) > \omega(\alpha_N(t), |x-y|)+1> \1_{E_t^{\gamma-}\cup\Gamma_t^\gamma}(x,x_d) - \1_{E_t^{\gamma-}}(y,y_d).
\end{equation}
The same of course holds if $y_d>1+LM$ and $x_d\leq 1+LM$.

Finally, suppose that $|x|>  r(M,\delta(t_0/2), t_0)+2$.  Then by Proposition \ref{p:limitlemma}, we have that for any $t\in [0,t_0]$ and $x_d\in \R$ that 
\begin{equation}
|U(t,x,x_d)-U_0(x,x_d)| < \frac{\delta(t_0/2)}{2}.
\end{equation}
In particular, 
\begin{equation}
\1_{E_t^{\gamma-}\cup\Gamma_t^\gamma}(x,x_d)=1 \quad \Rightarrow \quad x_d< \gamma+\frac{\delta(t_0/2)}{2}.
\end{equation}
Let $Y=(y,y_d)\in \R^d$.  Then either $|y|> r(M,\delta(t_0/2), t_0)$ or $|x-y|>2$.  If $|y|>r(M,\delta(t_0/2), t_0)$, then by the same argument $ \1_{E_t^{\gamma-}}(y,y_d) = 0$ implies $y_d >\gamma-\displaystyle\frac{\delta(t_0/2)}{2}$.  Hence, 
\begin{equation}
\1_{E_t^{\gamma-}\cup\Gamma_t^\gamma}(x,x_d) - \1_{E_t^{\gamma-}}(y,y_d)=1 < 1+\delta(t_0/2) -(x_d-y_d) < \Phi(t,x,x_d,y,y_d).
\end{equation}
On the other hand, if $|x-y|>2$, then $\omega(\alpha_N(t), |x-y|)>1+L|x-y|$.  Thus if $x_d-y_d\leq1+L|x-y|$, then by comparison principle and the definition of $\omega$
\begin{equation}
\1_{E_t^{\gamma-}\cup\Gamma_t^\gamma}(x,x_d) - \1_{E_t^{\gamma-}}(y,y_d)\leq 1 <(\omega(\alpha_{N}(t),|x-y|)-(x_d-y_d))+1\leq \Phi(t,x,x_d,y,y_d).
\end{equation}
And if $x_d-y_d>1+L|x-y|$, then by Lemma \ref{l:gammainitialmod}
\begin{equation}
\1_{E_t^{\gamma-}\cup\Gamma_t^\gamma}(x,x_d) - \1_{E_t^{\gamma-}}(y,y_d) = 0 <\Phi(t,x,x_d, y, y_d).
\end{equation}
Thus $|x|>r(M,\delta(t_0/2), t_0)+2$ implies that 
\begin{equation}
\1_{E_t^{\gamma-}\cup\Gamma_t^\gamma}(x,x_d) - \1_{E_t^{\gamma-}}(y,y_d)<\Phi(t,x,x_d,y,y_d).
\end{equation}
A symmetric argument clearly works in the case that $|y|>r(M,\delta(t_0/2), t_0)+2$.  
\end{proof}

Combining Lemmas \ref{l:phimodulus} and \ref{l:phibounds}, we now just have to show that 
\begin{equation}
\begin{split}
t\in [0,t_0], \ |x|,|y|\leq r(M,\delta(t_0/2),t_0), \ |x_d|,|y_d|\leq 1+LM 
\\ \Rightarrow \quad \Phi(t,x,x_d,y,y_d)\geq \1_{E_t^{\gamma-}\cup\Gamma_t^\gamma}(x,x_d) - \1_{E_t^{\gamma-}}(y,y_d).  
\end{split}
\end{equation}
To do this, we need to use our equations.  By Theorem \ref{t:subsuper} that $\1_{E_t^{\gamma-}\cup\Gamma_t^\gamma}(X)$ is a viscosity subsolution to \eqref{e:viscositysub} and that $\1_{E_t^{\gamma-}}(Y)$ is a viscosity supersolution to \eqref{e:viscositysuper}.  By standard viscosity solution arguments, it then follows that 
\begin{lemma}\label{l:viscositydouble}
The function $\1_{E_t^{\gamma-}\cup\Gamma_t^\gamma}(X)-\1_{E_t^{\gamma-}}(Y)$ is a viscosity subsolution to 
\begin{equation}\label{e:viscositydouble}
\begin{split}
\partial_t V(t,X,Y) \leq& -H_s(X, \{V(t, \cdot, Y)\geq V(t,X,Y)\})|\nabla_X V(t,X,Y)| 
\\&+ H_s(Y, \{V(t,X, \cdot)>V(t,X,Y)\})|\nabla_Y V(t,X,Y)|.
\end{split}
\end{equation}
In particular, if $\Psi$ crosses $\1_{E_t^{\gamma-}\cup\Gamma_t^\gamma}-\1_{E_t^{\gamma-}}$ at $(t,X,Y)$ with $\nabla_X\Psi,\nabla_Y\Psi\not = 0$, then for any $\epsilon>0$
\begin{equation}
\begin{split}
\partial_t \Psi(t,X,Y) \leq & s(1-s)|\nabla_X\Psi|\left(\int\limits_{|Z|>\epsilon} \frac{\1^\pm_{E_t^{\gamma -}\cup\Gamma_t^\gamma}(X+Z)}{|Z|^{d+s}}dZ  + P.V.\int\limits_{|Z|<\epsilon} \frac{\1^\pm_{\{\Psi(t,\cdot,Y)\geq \Psi(t,X,Y)\}}(X+Z)}{|Z|^{d+s}}dZ\right)
\\ &-s(1-s)|\nabla_Y \Psi|\left(\int\limits_{|Z|>\epsilon}\frac{\1^\pm_{E_t^{\gamma-}}(Y+Z)}{|Z|^{d+s}}dZ+ P.V.\int\limits_{|Z|<\epsilon} \frac{\1^\pm_{\{\Psi(t,X,\cdot)> \Psi(t,X,Y)\}}(Y+Z)}{|Z|^{d+s}}dZ\right).
\end{split}
\end{equation}
\end{lemma}

We leave the proof of Lemma \ref{l:viscositydouble} to the appendix.

Our plan now is to show that \eqref{e:phiinequality} can never hold.

\subsection{Proof of Lemma \ref{l:assumptionscase}}

We now prove that for any initial data $(E_0^-, \Gamma_0, E_0^+)$ satisfying \eqref{e:initialassumptions} and any $\gamma\in \left(\displaystyle\frac{-\eta}{\sqrt{1+L^2}}, \frac{\eta}{\sqrt{1+L^2}}\right)$ with
\begin{equation}
\mathcal{L}^{d}(\Gamma_t^\gamma) = 0 \text{ for almost every time } t\geq 0,
\end{equation}
that $\Gamma_t^\gamma$ has modulus $\omega(\alpha_N(t), \cdot)$ for times $t\in [0,t_0]$ where $t_0<2T$ is arbitrary.  By taking $N\to \infty$ and then $t_0\to 2T$, it then follows that $\Gamma_{2T}^\gamma$ will have modulus $\omega(T,\cdot)$ and thus be $(1+L)$-Lipschitz.  

By Lemma \ref{l:phimodulus}, it suffices to show that 
\begin{equation}
\begin{split}
\1_{E_t^{\gamma-}\cup\Gamma_t^\gamma}(x,x_d) - \1_{E_t^{\gamma-}}(y,y_d)&\leq \Phi(t,x,x_d,y,y_d)
\\&= \left(\omega(\alpha_N(t), |x-y|) -(x_d-y_d)+\frac{1}{2}\right)_+ +\frac{1}{2}
\end{split}
\end{equation}
for all $(t,X,Y)\in [0,t_0]\times \R^d\times \R^d$.  By Lemma \ref{l:phibounds}, we need only consider the case that 
\begin{equation}
t\in [0,t_0], \quad |x|,|y|\leq r(M,\delta(t_0/2), t_0)+2, \quad |x_d|,|y_d|\leq 1+LM.
\end{equation}
Recall that $\alpha_N'(t)\not = 0$ only for $t\in \displaystyle\bigcup\limits_{i=1}^N (a_i,b_i)$, where 
\begin{equation}\label{e:timedefn}
\displaystyle\bigcup\limits_i (a_i,b_i) = \mathcal{T}(t_0)= \left\{t\in (0,t_0):  \mathcal{L}^d (\Gamma_t^\gamma\cap B_{R(t_0)}^d) < \frac{\epsilon(t_0)^{d+s}}{8(2+L)s(1-s)T}\right\},
\end{equation}
and 
\begin{equation}
\epsilon(t_0):=\left(\frac{c\delta(t_0/2)^2}{8(2+L)sT}\right)^{1/(1-s)}, \qquad R(t_0) := r\left(M,\delta(t_0/2), t_0\right)+3+LM+ \left(\displaystyle\frac{1}{8(2+L)(1-s)T}\right)^{-1/s}.
\end{equation}

Without loss of generality, by reindexing we may assume that $0\leq a_1<b_1\leq a_2 < \ldots \leq b_N\leq t_0$.  It follows by the comparison principle and Lemma \ref{l:phibounds} that 
\begin{equation}
\begin{array}{c}
\1_{E_0^{\gamma-}\cup\Gamma_0^\gamma}(X) - \1_{E_0^{\gamma-}}(Y)< \Phi(0,X,Y) 
\\ \Rightarrow\quad \1_{E_t^{\gamma-}\cup\Gamma_t^\gamma}(X) - \1_{E_t^{\gamma-}}(Y)< \Phi(0,X,Y) = \Phi(t,X,Y), \quad t\in [0,a_1].
\end{array}
\end{equation}
By induction, suppose that $\1_{E_t^{\gamma-}\cup\Gamma_t^\gamma}(X) - \1_{E_t^{\gamma-}}(Y)< \Phi(t,X,Y)$ for $t\in [0,a_i]$.  
As $\1_{E_t^{\gamma-}\cup\Gamma_t^\gamma}(X) - \1_{E_t^{\gamma-}}(Y)$ is upper semicontinuous and $\Phi$ is continuous, by Lemma \ref{l:phibounds} we get that 
\begin{equation}
\1_{E_t^{\gamma-}\cup\Gamma_t^\gamma}(X) - \1_{E_t^{\gamma-}}(Y)< \Phi(t,X,Y) \text{ for } t\in (a_i,a_i+\epsilon), 
\end{equation}
for some $\epsilon>0$.  Thus there are two cases.  If  $\1_{E_t^{\gamma-}\cup\Gamma_t^\gamma}(X) - \1_{E_t^{\gamma-}}(Y)< \Phi(t,X,Y)$ for $t\in (a_i,b_i]$, then by the same argument as above 
\begin{equation}
\begin{array}{c}
\1_{E_{b_i}^{\gamma-}\cup\Gamma_{b_i}^\gamma}(X) - \1_{E_{b_i}^{\gamma-}}(Y)< \Phi(b_i,X,Y) 
\\ \Rightarrow\quad \1_{E_t^{\gamma-}\cup\Gamma_t^\gamma}(X) - \1_{E_t^{\gamma-}}(Y)< \Phi(b_i,X,Y) = \Phi(t,X,Y), \quad t\in [b_i,a_{i+1}].
\end{array}
\end{equation}

Else, we must that that $\Phi$ crosses $\1_{E_t^{\gamma-}\cup\Gamma_t^\gamma}- \1_{E_t^{\gamma-}}$ for some $t\in (a_i,b_i]$. 
If the crossing point was at time $b_i$, then by replacing $\Phi(t)$ with $\Phi(t-\epsilon)> \Phi(t)$ for some arbitrary $\epsilon>0$, we can regain the strict inequality.  Following the rest of the argument, we then get that at time $t_0$, $\Gamma_{t_0}^\gamma$ has modulus $\omega(\alpha_N(t_0-N\epsilon), \cdot)$ for an arbitrary $\epsilon>0$.  Hence, $\Gamma_{t_0}^\gamma$ has modulus $\omega(\alpha_N(t_0),\cdot)$.  

So without loss of generality, we may assume that any crossing point happens in the open interval $(a_i, b_i)\subset \mathcal{T}(t_0)$.  We will show that this is not possible,  proving that $\1_{E_t^{\gamma-}\cup\Gamma_t^\gamma}- \1_{E_t^{\gamma-}}< \Phi$ for all times $t\in [0,t_0]$ and thus completing the proof of Lemma \ref{l:assumptionscase} by Lemma \ref{l:phimodulus}.  

So now assume that $\Phi$ crosses $\1_{E_t^{\gamma-}\cup\Gamma_t^\gamma}- \1_{E_t^{\gamma-}}$ for some time $t\in (a_i,b_i)$.  By Lemma \ref{l:phibounds} we can thus assume our crossing point $(t,x,x_d,y,y_d)$ satisfies
\begin{equation}
t\in (a_i,b_i), \quad |x|,|y| \leq r(M,\delta(t_0/2), t_0)+2, \quad |x_d|,|y_d|\leq 1+LM.
\end{equation}
At the crossing point, we necessarily have that $\1_{E_t^{\gamma-}\cup\Gamma_t^\gamma}(x,x_d)- \1_{E_t^{\gamma-}}(y,y_d) = 1$.  The strict monotonicity of $\Phi$ in $x_d,y_d$ then implies that $x_d = \overline{u}^\gamma(t,x)$ and $y_d = \underline{u}^\gamma(t,y)$.  

By the definition of $\Phi$ \eqref{e:phidefn} we have that $\Phi$ is $C^1$ in time and $C^{1,1}$ in $X,Y$ on a neighborhood of the set $\Phi^{-1}(1)\subseteq [0,t_0]\times \R^d\times \R^d$.  Thus $\Phi$ is a valid test function for our purposes.  Taking $X= (x,\overline{u}^\gamma(t,x))$ and $Y = (y,\underline{u}^\gamma (t,y))$, applying Lemma \ref{l:viscositydouble} gives us that   
\begin{equation}\label{e:phiinequality}
\begin{split}
\partial_t \Phi(t,X,Y) \leq & s(1-s)|\nabla_X\Phi|\left(\int\limits_{|Z|>\epsilon(t_0)} \frac{\1^\pm_{E_t^{\gamma -}\cup\Gamma_t^\gamma}(X+Z)}{|Z|^{d+s}}dZ  + \int\limits_{|Z|<\epsilon(t_0)} \frac{\1^\pm_{\{\Phi(t,\cdot,Y)\geq \Phi(t,X,Y)\}}(X+Z)}{|Z|^{d+s}}dZ\right)
\\ &-s(1-s)|\nabla_Y \Phi|\left(\int\limits_{|Z|>\epsilon(t_0)}\frac{\1^\pm_{E_t^{\gamma-}}(Y+Z)}{|Z|^{d+s}}dZ+ \int\limits_{|Z|<\epsilon(t_0)} \frac{\1^\pm_{\{\Phi(t,X,\cdot)> \Phi(t,X,Y)\}}(Y+Z)}{|Z|^{d+s}}dZ\right)
\end{split}
\end{equation}
We will show that \eqref{e:phiinequality} is not possible, thus ruling out any crossing points.  

From the definition of $\Phi$ \eqref{e:phidefn}, we can immediately calculate that 
\begin{equation}\label{e:phitimebounds}
\partial_t\Phi(t,X,Y) = \frac{1}{2}\partial_t \omega(\alpha_N(t), |x-y|) >  \delta'(t/2) = \frac{-1}{T},
\end{equation}
and
\begin{equation}\label{e:phigradbounds}
\begin{array}{l}\ |\nabla_X \Phi(t,X,Y)| = |\nabla_Y \Phi(t,X,Y)| = \sqrt{1+\partial_r \omega(\alpha_N(t), |x-y|)^2},
\\ \Rightarrow \quad 1\leq |\nabla_X \Phi(t,X,Y)| = |\nabla_Y \Phi(t,X,Y)|\leq 2+L
\end{array}
\end{equation}

As $||\Phi(t,\cdot, Y)||_{\dot{W}^{2,\infty}}\leq ||\omega(\alpha_N(t),\cdot)||_{\dot{W}^{2,\infty}}\leq (c\delta(t_0/2)^2)^{-1}$,
 we have that 
\begin{equation}\label{e:pvphiestimate}
\begin{split}
\bigg|P.V.\int\limits_{|Z|<\epsilon(t_0)} \frac{\1^\pm_{\{\Phi(t,\cdot,Y)\geq \Phi(t,X,Y)\}}(X+Z)}{|Z|^{d+s}}dZ\bigg|&\leq \int\limits_{|z|<\epsilon(t_0)} \int\limits_{- (c\delta(t_0/2)^2)^{-1}|z|^2/2}^{ (c\delta(t_0/2)^2)^{-1}|z|^2/2} \frac{1}{(|z|^2+z_d^2)^{(d+s)/2}}dz_d dz
\\&\leq \int\limits_{|z|<\epsilon(t_0)} \frac{(c\delta(t_0/2)^2)^{-1}}{|z|^{d+s-2}}dz\leq \frac{\epsilon(t_0)^{1-s}}{(1-s)c\delta(t_0/2)^2}
\\&= \frac{1}{8(2+L)s(1-s)T}.
\end{split}
\end{equation}
The same holds for the $Y$ integral as well, so combining \eqref{e:pvphiestimate} with \eqref{e:phigradbounds} gives
\begin{equation}\label{e:pvphiestimatefinal}
\begin{split}
s(1-s)|\nabla_X&\Phi(t,X,Y)|\bigg|P.V.\int\limits_{|Z|<\epsilon(t_0)} \frac{\1^\pm_{\{\Phi(t,\cdot,Y)\geq \Phi(t,X,Y)\}}(X+Z)}{|Z|^{d+s}}dZ\bigg|
\\&+ s(1-s)|\nabla_Y\Phi(t,X,Y)|\bigg|P.V.\int\limits_{|Z|<\epsilon(t_0)} \frac{\1^\pm_{\{\Phi(t,X,\cdot)> \Phi(t,X,Y)\}}(Y+Z)}{|Z|^{d+s}}dZ\bigg| \leq \frac{1}{4T}.
\end{split}
\end{equation}

Plugging in \eqref{e:pvphiestimatefinal} and \eqref{e:phigradbounds} into \eqref{e:phiinequality} we get that 
\begin{equation}
\begin{split}
\partial_t \Phi(t,X,Y) \leq & \frac{1}{4T} + s(1-s)\sqrt{1+\partial_r\omega(\alpha_N(t),|x-y|)^2} \int\limits_{|Z|>\epsilon(t_0)} \frac{\1^\pm_{E_t^{\gamma -}\cup\Gamma_t^\gamma}(X+Z)-\1^\pm_{E_t^{\gamma-}}(Y+Z)}{|Z|^{d+s}}dZ  
\end{split}
\end{equation}

As 
\begin{equation}
\1^\pm_{E_t^{\gamma -}\cup\Gamma_t^\gamma}(X+Z)-\1^\pm_{E_t^{\gamma-}}(Y+Z) = 2\1_{\Gamma_t^\gamma}(X+Z) + \1^\pm_{E_t^{\gamma -}}(X+Z)-\1^\pm_{E_t^{\gamma-}}(Y+Z)
\end{equation}
we then have that 
\begin{equation}\label{e:timederivfirstbound}
\begin{split}
\partial_t \Phi(t,X,Y) \leq & \frac{1}{4T} + 2s(1-s)\sqrt{1+\partial_r\omega(\alpha_N(t),|x-y|)^2} \int\limits_{|Z|>\epsilon(t_0)} \frac{\1_{\Gamma_t^\gamma}(X+Z)}{|Z|^{d+s}}dZ  
\\&+s(1-s)\sqrt{1+\partial_r\omega(\alpha_N(t),|x-y|)^2} 
\int\limits_{|Z|>\epsilon(t_0)} \frac{\1^\pm_{E_t^{\gamma -}}(X+Z)-\1^\pm_{E_t^{\gamma-}}(Y+Z)}{|Z|^{d+s}}dZ.
\end{split}
\end{equation}
Letting $r_s = \left(\displaystyle\frac{1}{8(2+L)(1-s)T}\right)^{-1/s}$ and noting that 
\begin{equation}
|X|+r_s\leq r(M,\delta(t_0/2),t_0)+3+LM+r_s = R(t_0),
\end{equation}
 by Lemma \ref{l:phibounds}, we can bound the integral of the boundary term using \eqref{e:timedefn} as 
\begin{equation}\label{e:boundaryintbound}
\begin{split}
\int\limits_{|Z|>\epsilon(t_0)} \frac{\1_{\Gamma_t^\gamma}(X+Z)}{|Z|^{d+s}}dZ  &\leq \int\limits_{|Z|>r_s} \frac{1}{|Z|^{d+s}}dZ + \epsilon^{-(d+s)}\int\limits_{|Z|<r_s} \1_{\Gamma_t^\gamma}(X+Z)dZ 
\\& \leq \int\limits_{r_s}^\infty \frac{1}{r^{1+s}}dr + \epsilon(t_0)^{-(d+s)}\mathcal{L}^d(B_{r_s}^d(X)\cap \Gamma_t^\gamma)
\\& \leq \frac{1}{sr_s^{s}} + \epsilon(t_0)^{-(d+s)}\mathcal{L}^d(B_{R(t_0)}^d\cap \Gamma_t^\gamma)
\\& \leq \frac{1}{8(2+L)s(1-s)T} + \frac{1}{8(2+L)s(1-s)T} 
\\& = \frac{1}{4(2+L)s(1-s)T}.
\end{split}
\end{equation}

Plugging \eqref{e:boundaryintbound} back into \eqref{e:timederivfirstbound} then gives us
\begin{equation}\label{e:timederivsecondbound}
\partial_t \Phi(t,X,Y) \leq  \frac{3}{4T} +s(1-s)\sqrt{1+\partial_r\omega(\alpha_N(t),|x-y|)^2} 
\int\limits_{|Z|>\epsilon(t_0)} \frac{\1^\pm_{E_t^{\gamma -}}(X+Z)-\1^\pm_{E_t^{\gamma-}}(Y+Z)}{|Z|^{d+s}}dZ.
\end{equation}

Now the only thing that remains is to bound 
\begin{equation}
\int\limits_{|Z|>\epsilon(t_0)} \frac{\1^\pm_{E_t^{\gamma -}}(X+Z)-\1^\pm_{E_t^{\gamma-}}(Y+Z)}{|Z|^{d+s}}dZ.
\end{equation}
Applying Lemma \ref{l:fixedzestimate}, we get
\begin{equation}\label{e:intsetup}
\begin{split}
\int\limits_{|Z|>\epsilon(t_0)} &\frac{\1^\pm_{E_t^{\gamma -}}(X+Z)-\1^\pm_{E_t^{\gamma-}}(Y+Z)}{|Z|^{d+s}}dZ 
\\&\leq \int\limits_{|z|>\epsilon(t_0)} \int\limits_{\underline{u}^\gamma\left(t,x+z\right)-\overline{u}^\gamma\left(t,x\right)}^{\overline{u}^\gamma\left(t,y+z\right)-\underline{u}^\gamma\left(t,y\right)} \frac{\1^\pm_{E_t^{\gamma -}}(x+z,\overline{u}^\gamma(t,x)+z_d)-\1^\pm_{E_t^{\gamma-}}(y+z,\underline{u}^\gamma(t,y)+z_d)}{(|z|^2+\omega(\alpha_N(t),|z|)^2)^{(d+s)/2}}dz_d dz 
\\&= \int\limits_{\R^{d-1}}  \int\limits_{\underline{u}^\gamma\left(t,x+z\right)-\overline{u}^\gamma\left(t,x\right)}^{\overline{u}^\gamma\left(t,y+z\right)-\underline{u}^\gamma\left(t,y\right)}\frac{\1^\pm_{E_t^{\gamma -}}(x+z,\overline{u}^\gamma(t,x)+z_d)-\1^\pm_{E_t^{\gamma-}}(y+z,\underline{u}^\gamma(t,y)+z_d)}{(|z|^2+\omega(\alpha_N(t),|z|)^2)^{(d+s)/2}}dz_d dz 
\\& \quad -\int\limits_{|z|<\epsilon(t_0)}  \int\limits_{\underline{u}^\gamma\left(t,x+z\right)-\overline{u}^\gamma\left(t,x\right)}^{\overline{u}^\gamma\left(t,y+z\right)-\underline{u}^\gamma\left(t,y\right)} \frac{\1^\pm_{E_t^{\gamma -}}(x+z,\overline{u}^\gamma(t,x)+z_d)-\1^\pm_{E_t^{\gamma-}}(y+z,\underline{u}^\gamma(t,y)+z_d)}{(|z|^2+\omega(\alpha_N(t),|z|)^2)^{(d+s)/2}}dz_d dz 
\end{split}
\end{equation}
We can bound that last error term rather easily by 
\begin{equation}
\begin{split}
-\int\limits_{|z|<\epsilon(t_0)} \int\limits_{\underline{u}^\gamma\left(t,x+z\right)-\overline{u}^\gamma\left(t,x\right)}^{\overline{u}^\gamma\left(t,y+z\right)-\underline{u}^\gamma\left(t,y\right)}& \frac{\1^\pm_{E_t^{\gamma -}}(x+z,\overline{u}^\gamma(t,x)+z_d)-\1^\pm_{E_t^{\gamma-}}(y+z,\underline{u}^\gamma(t,y)+z_d)}{(|z|^2+\omega(\alpha_N(t),|z|)^2)^{(d+s)/2}}dz_d dz 
\\&\leq \int\limits_{|z|<\epsilon(t_0)} \int\limits_{-\omega(\alpha_N(t), |z|)}^{\omega(\alpha_N(t), |z|)} \frac{2}{\omega(\alpha_N(t),0)^{d+s}}dz_d dz 
\\&\leq \frac{4\omega(\alpha_N(t), \epsilon(t_0))\mathcal{L}^{d-1}(B_1^{d-1}) }{\delta(t_0/2)^{d+s}}\epsilon(t_0)^{d-1}
\\& = \frac{4\omega(\alpha_N(t), \epsilon(t_0))\mathcal{L}^{d-1}(B_1^{d-1}) }{\delta(t_0/2)^{d+s}} \left(\frac{c\delta(t_0/2)^2}{16sT}\right)^{(d-1)/(1-s)}
\\& \lesssim_{d,s,L,M}  \delta(t_0/2)^{d-1-s}
\end{split}
\end{equation}
Thus as long as we take $t_0$ sufficiently close to $2T$ and thus $\delta(t_0/2)$ sufficiently small, we can guarantee that 
\begin{equation}\label{e:epsilonbound}
\begin{split}
-\int\limits_{|z|<\epsilon(t_0)} \int\limits_{\underline{u}^\gamma\left(t,x+z\right)-\overline{u}^\gamma\left(t,x\right)}^{\overline{u}^\gamma\left(t,y+z\right)-\underline{u}^\gamma\left(t,y\right)} &\frac{\1^\pm_{E_t^{\gamma -}}(x+z,\overline{u}^\gamma(t,x)+z_d)-\1^\pm_{E_t^{\gamma-}}(y+z,\underline{u}^\gamma(t,y)+z_d)}{(|z|^2+\omega(\alpha_N(t),|z|)^2)^{(d+s)/2}}dz_d dz 
\\&\leq \frac{1}{4(2+L)s(1-s)T}.
\end{split}
\end{equation}
Since we care about that limit as $t_0\to 2T$, this isn't an issue.  

Plugging \eqref{e:intsetup} and \eqref{e:epsilonbound} into \eqref{e:timederivsecondbound} we get that 
\begin{equation}
\begin{split}
\ &\frac{\partial_t \Phi(t,X,Y) -\frac{1}{T} }{s(1-s)\sqrt{1+\partial_r\omega(\alpha_N(t),|x-y|)^2}} 
\\  &\qquad \leq \int\limits_{\R^{d-1}} \int\limits_{\underline{u}^\gamma\left(t,x+z\right)-\overline{u}^\gamma\left(t,x\right)}^{\overline{u}^\gamma\left(t,y+z\right)-\underline{u}^\gamma\left(t,y\right)} \frac{\1^\pm_{E_t^{\gamma -}}(x+z,\overline{u}^\gamma(t,x)+z_d)-\1^\pm_{E_t^{\gamma-}}(y+z,\underline{u}^\gamma(t,y)+z_d)}{(|z|^2+\omega(\alpha_N(t),|z|)^2)^{(d+s)/2}}dz_d dz 
\end{split}
\end{equation}

Now after all of this setup, we have essentially returned to smooth case.  Indeed, all of our integral bounds in sections 3 and 4 never used that the flow $t\to E_t$ was smooth.  So we can apply Lemmas \ref{l:P.V.estimate} through \ref{l:omeganegintbound} to get that just as in the smooth case, 
\begin{equation}
\begin{split}
s(1-s)\sqrt{1+\partial_r\omega(\alpha_N(t),|x-y|)^2} &
\int\limits_{\R^{d-1}}\int\limits_\R \frac{\1^\pm_{E_t^{\gamma -}}(x+z,\overline{u}^\gamma(t,x)+z_d)-\1^\pm_{E_t^{\gamma-}}(y+z,\underline{u}^\gamma(t,y)+z_d)}{(|z|^2+\omega(\alpha_N(t),|z|)^2)^{(d+s)/2}}dz_d dz
\\& \leq\frac{-2}{T}.
\end{split}
\end{equation}
Recalling \eqref{e:phitimebounds}, we thus have that 
\begin{equation}
\partial_t \Phi(t,X,Y) \leq \frac{1}{T}-\frac{2}{T} = \frac{-1}{T} =\delta'(t/2)<\alpha'_N(t)\partial_t\omega(\alpha_N(t),|x-y|) = \partial_t\Phi(t,X,Y),
\end{equation}
a contradiction.  Thus we could not have a crossing point for times $t\in (a_i,b_i)$.  By induction, it then follows that there is no crossing point for any time $t\in [0,t_0]$, so 
\begin{equation}
\1_{E_t^{\gamma-}\cup\Gamma_t^\gamma}(X) - \1_{E_t^{\gamma-}}(Y)\leq \Phi(t,X,Y), \qquad \text{ for all } (t,X,Y)\in [0,t_0]\times \R^d\times \R^d.
\end{equation}
Hence, $\Gamma_t^\gamma$ has modulus $\omega(\alpha_N(t),\cdot)$ for all times $t\in [0,t_0].$  Letting $N\to \infty$ and $t_0\to 2T$, we then have that $\Gamma_t^\gamma$ has modulus $\omega(t/2, \cdot)$ for all times $t\in [0,2T]$.  Thus $\Gamma_{2T}^\gamma$ is $(1+L)$-Lipschitz.


\section{Proofs by Approximations}

\begin{corollary}\label{c:decaycase}
Let $(E_0^-, \Gamma_0, E_0^+)$ satisfy \eqref{e:initialassumptions}.  Then for every $\gamma\in \left(\displaystyle\frac{-\eta}{\sqrt{1+L^2}},\frac{\eta}{\sqrt{1+L^2}}\right)$, $\partial E_{2T}^{\gamma\pm}$ is $(1+L)$-Lipschitz.  
\end{corollary}
\begin{proof}
Without loss of generality, we will simply prove in the case that $\gamma = 0$ that $\partial E_{2T}^-$ is $(1+L)$-Lipschitz.  Let $X_0\in \partial  E_{2T}^-$ be arbitrary, and consider a sequence $(\gamma_n)_{n=1}^\infty$ such that  $\displaystyle\frac{-\eta}{\sqrt{1+L^2}}<\gamma_1<\gamma_2 <\ldots  <\gamma_n\to 0$ and $\gamma_n$ satisfies
\begin{equation}
\mathcal{L}^d(\Gamma_t^{\gamma_n}) = 0, \quad \text{ for almoset every }t.
\end{equation}
For each $n\in N$, let $X_n\in \Gamma_{2T}^{\gamma_n}$ be a point closest to $X_0$.  Then 
\begin{equation}
|X_n-X_0|\geq |X_{n+1}-X_0|\to 0.
\end{equation}
By Lemma \ref{l:assumptionscase},  $\Gamma_{2T}^{\gamma_n} = \text{graph}(u^{\gamma_n}:\R^{d-1}\to \R)$ is a $(1+L)$-Lipschitz graph with $\overline{E^{\gamma_n-}_{2T}} = \{(x,x_d)| x_d\leq u^{\gamma_n}(x)\}$, it follows that
\begin{equation}
\{X_n-(z, z_d): z\in \R^{d-1},  z_d\geq (1+L)|z|\}\subseteq \overline{E_{2T}^{\gamma_n -}}\subseteq E_{2T}^-.
\end{equation}
Thus taking the union we get
\begin{equation}
\{X_0-(z,z_d): z\in \R^{d-1}, z_d>(1+L)|z|\}\subseteq \bigcup\limits_{n=1}^\infty \{X_n-(z, z_d): z\in \R^{d-1}, z_d\geq (1+L)|z|\} \subseteq E_{2T}^-.
\end{equation}
As $X_0\in \partial E_{2T}^-$ was arbitrary, we thus have that $\partial E_{2T}^-$ is a $(1+L)$-Lipschitz graph.  
\end{proof}

We've now proven Theorem \ref{t:main} under the assumptions that 
\begin{equation}
\left\{\begin{array}{l} 1). E_0^\pm \text{ have modulus } (1-\eta)+Lr,
\\ 2). 0 \in \Gamma_0
\\ 3). \Gamma_0 \setminus (B_M^{d-1}\times \R) = \{(x,0): x\in \R^{d-1}, |x|\geq M\},
\end{array}\right.
\end{equation}
for some $0<\eta<<1$ and $1<<M<\infty$.  Our next goal is to remove the flatness assumption and allow more arbitrary behavior at infinity by letting $M\to \infty$, but in order to justify that we need some compactness.  

\begin{lemma}\label{l:compactness}
(Compactness) Let $U_0: \R^d\to \R$ be 1-Lipschitz.  Then the unique viscosity solution $U: [0,\infty)\times \R^d\to \R$ is 1-Lipschitz in space with 
\begin{equation}
\sup\limits_{X\in \R^d}\sup\limits_{t,t'>0} \frac{|U(t,X)-U(t',X)|}{|t-t'|^{1/(1+s)}}\leq \left[(1+s)H_s(B_1^d)\right]^{1/(1+s)}.
\end{equation}
\end{lemma}
\begin{proof}
By Theorem \ref{t:comparison} we have that the viscosity solution $U(t,\cdot)$ will be 1-Lipschitz for all times $t$, so we only need to prove the $C^{1/(1+s)}$ estimate in time.  Without loss of generality, assume that $t<t'$,$U(t',X)=0$, and $U(t,X)> 0$.  As $U(t,\cdot)$ is 1-Lipschitz, we have that 
\begin{equation}
d(X,\Gamma_t)\geq U(t,X).
\end{equation}
Thus $B^d(X,U(t,X))\subseteq E_t^+$.  Let $H_s(B_1^d)$ be the $s$-fractional mean curvature of $B_1^d$.  Then 
\begin{equation}
\tau \to B^d(X,\left[U(t,X)^{1+s}-(1+s)H_s(B_1^d)\tau\right]^{1/(1+s)})
\end{equation}
is smooth fractional mean curvature flow for $0\leq \tau< \frac{U(t,X)^{1+s}}{(1+s)H_s(B_1^d)}$.  Thus by comparison principle (Proposition \ref{p:comparisonprinciple}), we then have that $X\in E^{+}_{t+\tau}$ for all such $\tau$.  As $t'>t$ and $U(t',X) =  0$, we thus have that $|t'-t|\geq \frac{U(t,X)^{1+s}}{(1+s)H_s(B_1^d)}$.  Thus 
\begin{equation}
\sup\limits_{t'>t} \frac{|U(t,X)-U(t',X)|}{|t-t'|^{1/(1+s)}}\leq \left[(1+s)H_s(B_1^d)\right]^{1/(1+s)}.
\end{equation}
\end{proof}

\subsection{Proof of Theorem \ref{t:main}}

\begin{proof}
It suffices to show that if $(E_0^-, \Gamma_0, E_0^+)$ is such that $E_0^{\pm}$ have modulus of continuity $1-\eta+Lr$ for some $\eta>0$, then $\partial E_{t}^{\pm}$ are $(1+L)$-Lipschtiz graphs for $t\geq 2T(d,s,L)$.    

Without loss of generality, we may assume that $0\in \Gamma_0$.  To begin, let $U$ be the viscosity solution of the level set equation \eqref{e:viscosity} for the initial data $U_0$ as in \eqref{e:U0}, the signed distance function of $\Gamma_0$.

For a set $A\subseteq \R^d$ and $x\in \R^{d-1}$, let $A(x) = \{x_d| (x,x_d)\in A\}$.  Then for $M>>1$, we define the sets $E^{\pm(M)}_0$ by 
\begin{equation}
E_0^{\pm}(x) = \left\{\begin{array}{ll} E_0^{\pm}(x), & |x|\leq M \\ E_0^{\pm}(M\hat{x}), & M\leq |x|\leq M^2/2, \\ \frac{2(M^2-|x|)}{M^2}E_0^{\pm}(M\hat{x}), & M^2/2\leq |x|\leq M^2, \\ \{\pm x_d>0\}, & |x|\geq M^2 \end{array}\right. 
\end{equation}
where $\hat{x} = \displaystyle\frac{x}{|x|}$.  The sets $E_0^\pm$ are open and disjoint, and have modulus of continuity $1-\frac{\eta}{2}+Lr$ for $M$ sufficiently large.  Taking $\Gamma_0^{(M)} = \R^d\setminus (E_0^{-(M)} \cup E_0^{+(M)})$ and $U_0^{(M)}$ to be the signed distance function of $\Gamma_0^{(M)}$, we then have a unique viscosity solution $U^{(M)}(t,X)$ with $U^{(M)}(0,0) = 0$.

By Lemma \ref{l:compactness} we have that $(U^{(M)}(\cdot, \cdot))_{M}$ is precompact in $C_{loc}([0,\infty)\times \R^d)$.  By classical viscosity solution arguments, we have that any limit will also be a viscosity solution.  As $\lim\limits_{M\to \infty} U^{(M)}_0(X) = U_0(X)$, we have by the uniqueness of viscosity solutions that there is a sequence $M_k\to \infty$ such that $U^{(M_k)}(t,X)\to U(t,X)$ in $C_{loc}([0,\infty)\times \R^d)$.  

Now fix some time $t\geq 2T(d,s,L)$ and point $X_0\in \partial E_{t}^- = \partial\{U(t,\cdot)<0\}$.  Let $(\gamma_n)_{n=1}^\infty$ be such that $\displaystyle\frac{-\eta}{2\sqrt{1+L^2}}<\gamma_1<\gamma_2 <\ldots  <\gamma_n\to 0$, and let $X_n\in \Gamma_t^{\gamma_n} = \{U(t,\cdot)=\gamma_n\}$ be a point closest to $X_0$.  Without loss of generality, we can assume that $|X_1-X_0|\leq 1$, and hence
\begin{equation}
1\geq |X_n-X_0|\geq |X_{n+1}-X_0|\to 0.
\end{equation}
As $U^{(M_k)}(t,\cdot)\to U(t,\cdot)$ in $C_{loc}$, we can find a $k(n)\in \N$ such that 
\begin{equation}
|U^{(M_{k(n)})}(t,X)- U(t,X)| <-\gamma_n/2, \ |X-X_0|\leq n \quad \Rightarrow \quad -\eta<2\gamma_n<U^{(M_{k(n)})}(t,X_n) <\gamma_n/2.
\end{equation}
By Corollary \ref{c:decaycase} applied to level sets of $U^{(M_k)}$, we have that 
\begin{equation}
\{X_n-(z, z_d): z\in \R^{d-1},z_d\geq (1+L)|z|\}\cap B_{n}^d(X_0)\subseteq \{U^{(M_{k(n)})}(t,\cdot)\leq \gamma_n/2\}\cap B_{n}^d(X_0)\subseteq E_{t}^-.
\end{equation}
Thus taking the union it follows that 
\begin{equation}
\begin{split}
\{X_0-(z,z_d):& z\in \R^{d-1}, z_d>(1+L)|z|\}
\\&\subseteq \bigcup\limits_{n=1}^\infty \left(\{X_n-(z, z_d): z\in \R^{d-1},  z_d\geq (1+L)|z|\}\cap B_n^d(X_0)\right) \subseteq E_{t}^-.
\end{split}
\end{equation}
As $X_0\in \partial E_{t}^-$ was arbitrary, we thus have that $\partial E_{t}^-$ is a $(1+L)$-Lipschitz graph.  
\end{proof}


\begin{appendix}
\section{Viscosity Solutions}

In this appendix, we review the necessary definitions and essential existence/uniqueness results for defining weak solutions via the level set method.  For more details, we refer the reader to \cite{CyrilFract, Chambolle, CrystallineChambolle}, or \cite{EvansMCF1} for details about level set method for classical mean curvature flow.  

\begin{definition}\label{d:viscosity}
\hfill

\noindent 1).  An upper semicontinuous function $U:(0,T)\times \R^d\to \R$ is a viscosity subsolution to the level set equation 
\begin{equation}\label{e:viscositysub}
\partial_t U(t,X) \leq -H_s(X, \{U(t,\cdot)\geq U(t,X)\})|\nabla_XU(t,X)|
\end{equation}
if whenever $\Psi$ is a smooth test function such that crosses $U$ from above at $(t,X)$, then $\partial_t \Psi(t,X)\leq 0$ if $\nabla_X\Psi(t,X) = 0$ or else
\begin{equation}
\partial_t \Psi(t,X) \leq \left(\int\limits_{|Z|>\epsilon} \frac{\1^\pm_{\{U(t,\cdot)\geq U(t,X)\}}(X+Z)}{|Z|^{d+s}}dZ  + \int\limits_{|Z|<\epsilon} \frac{\1^\pm_{\{\Psi(t,\cdot)\geq \Psi(t,X)\}}(X+Z)}{|Z|^{d+s}}dZ\right)|\nabla_X\Psi(t,X)|,
\end{equation}
for any $\epsilon>0$.  

\bigskip
\noindent 2).  A lower semicontinuous function $U:(0,T)\times\R^d\to \R$ is a viscosity supersolution to the level set equation

\begin{equation}\label{e:viscositysuper}
\partial_t U(t,X) \geq -H_s(X, \{U(t,\cdot)> U(t,X)\})|\nabla_XU(t,X)|,
\end{equation}
if whenever $\Psi$ is a smooth test function such that crosses $U$ from below at $(t,X)$, then  $\partial_t \Psi(t,X)\geq 0$ if $\nabla_X\Psi(t,X) = 0$ or else
\begin{equation}
\partial_t \Psi(t,X) \geq \left(\int\limits_{|Z|>\epsilon} \frac{\1^\pm_{\{U(t,\cdot)> U(t,X)\}}(X+Z)}{|Z|^{d+s}}dZ  + \int\limits_{|Z|<\epsilon} \frac{\1^\pm_{\{\Psi(t,\cdot)> \Psi(t,X)\}}(X+Z)}{|Z|^{d+s}}dZ\right)|\nabla_X\Psi(t,X)|,
\end{equation}
for any $\epsilon>0$.  

3).  We say that a continuous function $U: (0,T)\times \R^d\to \R$ is a viscosity solution to the level set equation
\begin{equation}\label{e:viscosity}
\partial_t U(t,X) = -H_s(X, \{U(t,\cdot)\geq U(t,X)\})|\nabla_XU(t,X)|,
\end{equation}
if $U$ is both a subsolution and a supersolution.  

\end{definition}

\begin{theorem}\label{t:comparison} \cite{CyrilFract} Let $U_0\in \dot{W}^{1,\infty}(\R^d)$, and suppose that $U,V: [0,T]\times \R^d\to \R$ are sub/super solutions respectively to \eqref{e:viscositysub},\eqref{e:viscositysuper} with $U(0,X)\leq U_0(X)\leq V(0,X)$.  Then $U(t,X)\leq V(t,X)$ for all $(t,X)\in [0,T]\times \R^d$.  

In particular, there is a unique viscosity solution $U:[0,\infty)\times\R^d\to \R$ to \eqref{e:viscosity} with $U(0,X)=U_0(X)$.  Furthermore, 
\begin{equation}
||\nabla_XU||_{L^\infty([0,\infty)\times\R^d)} = ||\nabla_X U_0||_{L^\infty(\R^d)}.
\end{equation}
The sets 
\begin{equation}
\{U(t,\cdot)<0\}, \quad \{U(t,\cdot)=0\}, \quad \{U(t,\cdot)>0\}
\end{equation}
depend only on the initial sets
\begin{equation}
\{U_0<0\}, \qquad \{U_0 = 0\}, \quad \{U_0>0\}.
\end{equation}
\end{theorem}

\begin{definition}\label{d:viscosityset}
Let $(E_0^-, \Gamma_0, E_0^+)$ be such that $E_0^\pm$ are open, $\Gamma_0$ is closed, all are mutually disjoint with $E_0^-\cup\Gamma_0\cup E_0^+ = \R^d$.  Then we define the viscosity solution to fractional mean curvature flow \eqref{e:meancurvfloweqnset} as follows.  Let $U_0\in \dot{W}^{1,\infty}(\R^d)$ be such that that 
\begin{equation}
E_0^- = \{U_0<0\}, \quad \Gamma_0 = \{U_0 = 0\}, \quad E_0^+ = \{U_0>0\},
\end{equation}
and $U(t,X)$ be the unique viscosity solution to \eqref{e:viscosity} with initial data $U_0$.  Then the sets 
\begin{equation}
E_t^- = \{U(t,\cdot)<0\}, \quad \Gamma_t= \{U(t,\cdot) = 0\}, \quad E_t^+ = \{U(t,\cdot)>0\},
\end{equation}
are independent of the choice of $U_0$.   We call $t\to (E_t^-, \Gamma_t, E_t^+)$ the viscosity solution to fractional mean curvature flow. The flows $t\to E_t^-\cup\Gamma_t$ and $t\to E_t^-$ are called the maximal subsolution and minimal superoslution of the flow \eqref{e:meancurvfloweqnset} respectively.   
\end{definition}

The names maximal subsolution and minimal supersolution are quite natural for $E_t^-\cup \Gamma_t$ and $E_t^-$, as in fact

\begin{theorem} \label{t:subsuper}\cite{CyrilFract}
Let $(E_t^-, \Gamma_t, E_t^+)$ be the viscosity solution of fractional mean curvature flow for the initial data $(E_0^-, \Gamma_0, E_0^+)$.  Then the indicator functions $\1_{E_t^-\cup \Gamma_t},\1_{E_t^-}$ are the maximal subsolution  and minimal supersolution to \eqref{e:viscositysub},\eqref{e:viscositysuper} for the initial data $U_0 = \1_{E_0^-\cup \Gamma_0}$ and $U_0= \1_{E_0^-}$ respectively.  
\end{theorem}

While the comparison principle on the level of sets is proven in \cite{CyrilFract}, for our purposes it is more useful to apply it to just the minimal superosolutions.  Hence, 

\begin{proposition}\label{p:comparisonprinciple} (Comparison Principle)
Let $E_0, F_0\subseteq \R^d$ be open sets with $E_0\subseteq F_0$, and $E_t, F_t$ be the minimal viscosity supersolutions of the flow.  Then $E_t\subseteq F_t$ for all times $t\geq 0$.  
\end{proposition}

\begin{proof}
Note that as $E_0\subseteq F_0$, 
\begin{equation}
\1_{E_0}(X)\leq \1_{F_0}(X).  
\end{equation}
Thus as $\1_{E_t}$ is the minimal supersolution of \eqref{e:viscositysuper} with respect to the initial data $U_0 = \1_{E_0}$ and $\1_{F_t}$ is a supersolution, we have by minimality that 
\begin{equation}
\1_{E_t}(X)\leq \1_{F_t}(X),
\end{equation}
for all $t$, and hence $E_t\subseteq F_t$.  
\end{proof}

\begin{proposition}\label{p:fracmeancurvbasics} Basic properties of fractional mean curvature:

Let $E\subseteq \R^d$ with $0\in \partial E$ and $H_s(0,E)$ well defined.  

\begin{enumerate}
\item Translation invariance: for any $Y\in \R^d$, $H_s(Y, E+Y) = H_s(0,E)$.  

\item Symmetry: $H_s(0,-E) = H_s(0,E)$

\item Scaling: for any $r>0$, $r^sH_s(0, rE) =H_s(0,E)$.  

\item Monotonicty: if $E\subseteq F$ is with $0\in \partial F$, then $H_s(0,F)\leq H_s(0,E)$.  
\end{enumerate}

\end{proposition}
\begin{proof}
1-3. follow by making a simple change of variables to 
\begin{equation}
P.V.\int\limits_{\R^d} \frac{\1_E^\pm(Z)}{|Z|^{d+s}}dZ.
\end{equation}
4. follows from noting that $\1^\pm_E(Z)\leq \1^\pm_F(Z)$ for all $Z\in \R^d$.  
\end{proof}

\begin{corollary}\label{c:fracmeancurvflowrescaling}
Let $U(t,X)$ be a viscosity subsolution to \eqref{e:viscositysub}.  Then so is 
\begin{enumerate}
\item $U(t,X+Y)$ for fixed $Y\in \R^d$,
\item $U(t,-X)$,
\item $U(r^{1+s}t, rX)$ for $r>0$.
\end{enumerate}
\end{corollary}

\begin{lemma} \label{l:fractmeancurvsemicont}
Suppose that $F_n\in C_{loc}((0,T); W^{2,\infty}(\R^d))$ is uniformly bounded and $F_n\to F$ in $C_{loc}((0,T)\times \R^d)$.  Then for any $(t,X)\in (0,T)\times \R^d$, 
\begin{equation}
\limsup\limits_{n\to \infty} -H_s(X,\{F_n(t,\cdot)\geq F_n(t,X)\})\leq -H_s(X,\{F(t,\cdot)\geq F(t,X)\})
\end{equation}

In particular, for fixed $F\in C((0,T); W^{2,\infty}(\R^d))$ the function 
\begin{equation}
(t,X)\to -H_s(X, \{F(t,\cdot)\geq F(t,X)\}),
\end{equation}
is upper semicontinuous.  
\end{lemma}

\begin{proof}
Fix some point $(t, X)\in (0,T)\times \R^d$ and let $\epsilon>0$.  As $(F_n)_{n=1}^\infty$ is uniformly $C^{1,1}$ in space for times $t'\in (\frac{t}{2}, \frac{t+T}{2})$, we have that 
\begin{equation}
\bigg| s(1-s)\int\limits_{|Z|<r} \frac{\1^\pm_{\{F_n(t,\cdot)\geq F_n(t,X)\}}(X+Z)}{|Z|^{d+s}}dZ\bigg| \lesssim r.
\end{equation}
uniformly in $n$.  Taking $r_0\lesssim \epsilon$, we thus have that 
\begin{equation}\label{e:smallr0bound}
\bigg|s(1-s)\int\limits_{|Z|<r_0} \frac{\1^\pm_{\{F_n(t,\cdot)\geq F_n(t,X)\}}(X+Z)}{|Z|^{d+s}}dZ\bigg| \leq \frac{\epsilon}{5}, 
\end{equation}
for all $n$.  

Now take $R_0>0$ large enough so that
\begin{equation}\label{e:largeR0bound}
\bigg|s(1-s)\int\limits_{\R^d\setminus B_{R_0}^d} \frac{\1^\pm_{\{F_n(t,\cdot)\geq F_n(t,X)\}}(X+Z)}{|Z|^{d+s}}dZ\bigg| \leq s(1-s)\int\limits_{\R^d\setminus B_{R_0}^d} \frac{1}{|Z|^{d+s}}dZ \leq \frac{\epsilon}{5}
\end{equation}
Combining \eqref{e:smallr0bound} and \eqref{e:largeR0bound} gives us that 
\begin{equation}\label{e:annulusbound}
\begin{split}
-H_s &(X,  \{F_n(t,\cdot\geq F_n(t,X)\}) + H_s(X, \{F(t,\cdot)\geq F(t, X)\}) 
\\&\leq \frac{4}{5}\epsilon + s(1-s)\int\limits_{B_{R_0}^d\setminus B_{r_0}^d} \frac{\1^\pm_{\{F_n(t,\cdot)\geq F_n(t,X)\}}(X+Z)-\1^\pm_{\{F(t,\cdot)\geq F(t,X)\}}(X+Z)}{|Z|^{d+s}}dZ.  
\end{split}
\end{equation}

Note that as $\eta\to 0+$, we have the convergence
\begin{equation}
 \1^\pm_{\{F(t,\cdot)\geq F(t,X)-\eta\}} \xrightarrow[]{L^1_{loc}} \1^\pm_{\{F(t,\cdot)\geq F(t,X)\}}.
\end{equation}
Hence for some $\eta_0>0$ sufficiently small,
\begin{equation}\label{e:eta0small}
s(1-s)\int\limits_{B_{R_0}^d\setminus B_{r_0}^d} \frac{ \1^\pm_{\{F(t,\cdot)\geq F(t,X)-\eta_0\}}-\1^\pm_{\{F(t,\cdot)\geq F(t,X)\}}(X+Z)}{|Z|^{d+s}}dZ < \frac{\epsilon}{5}.
\end{equation}
As $F_n\to F$ in $C_{loc}$, we have that for $n$ sufficiently large that
\begin{equation}
\{Z: F_n(t,X+Z)\geq F_n(t,X)\} \cap B_{R_0}^d \subseteq \{Z: F(t, X+Z)\geq F(t,X)-\eta_0\}\cap B_{R_0}^d.  
\end{equation}
In particular, we then have that 
\begin{equation}\label{e:signedcharacteristicinequality}
\1^\pm_{\{F_n(t,\cdot)\geq F(t,X)\}}(X+Z) \leq \1^\pm_{\{F(t, \cdot)\geq F(t,X)-\eta_0\}}(X_0+Z), \qquad |Z|\leq R_0.
\end{equation}
Plugging in \eqref{e:signedcharacteristicinequality} and \eqref{e:eta0small} into \eqref{e:annulusbound} then gives us that 
\begin{equation}
-H_s(X,  \{F(t,\cdot\geq F(t,X)\}) + H_s(X_0, \{F(t_0,\cdot)\geq F(t_0, X_0)\}) <\epsilon,
\end{equation}
whenever $n$ is sufficiently large.    
\end{proof}

\begin{lemma}\label{l:viscositydoublesmooth}
Let $U_1, U_2\in L^\infty((0,T); W^{1,\infty}( \R^d))\cap L^\infty (\R^d; C^{1/(1+s)}(0,T))$ with $U_1$ a subsolution to \eqref{e:viscositysub} and $U_2$ a viscosity supersoltuion to \eqref{e:viscositysuper}.  Then $U_1(t,X)-U_2(t,Y)$ is a viscosity subsolution to \eqref{e:viscositydouble}.  
\end{lemma}

\begin{proof}
Our goal is show that $V(t,X,Y) = U_1(t,X)-U_2(t,Y)$ is a subsolution of \eqref{e:viscositydouble}.  

Suppose $\Psi: (0, T)\times \R^d\times \R^d\to \R$ satisfies 
\begin{equation}
\left\{\begin{array}{ll}
\Psi(t_0,X_0,Y_0) = V(t_0,X_0,Y_0), & \\  \Psi(t,X,Y) > V(t,X,Y), &  (t,X,Y)\not = (t_0, X_0, Y_0) \end{array}\right. .
\end{equation}

In fact, without loss of generality we may assume that 
\begin{equation}
\Psi(t,X,Y) - (U_1(t,X)-U_2(t,Y)) \geq \alpha \min\{|t-t_0|^2+|X-X_0|^2+|Y-Y_0|^2, 1\},
\end{equation}
for some $\alpha>0$ arbitrarily small.

Consider $\Psi_\epsilon(t,s, X, Y) = \displaystyle\Psi(\frac{t+s}{2}, X,Y) + \frac{|t-s|^2}{2\epsilon}$.  Then 
\begin{equation}\label{e:viscositylemmalowerbound}
\begin{split}
\Psi_\epsilon&(t,s,X,Y) - (U_1(t,X)-U_2(s, Y))
\\&\geq  \displaystyle\Psi(\frac{t+s}{2}, X,Y) + \frac{|t-s|^2}{2\epsilon} - (U_1(\frac{t+s}{2}, X) - U_2((\frac{t+s}{2}, Y)) - \frac{|t-s|^{\frac{1}{1+s}}}{2}(|| U_1||_{L^\infty_XC^{\frac{1}{1+s}}_t} + ||U_2||_{L^\infty_YC^{\frac{1}{1+s}}_t} ) 
\\& \geq \alpha \min\{|\frac{t+s}{2}-t_0|^2+|X-X_0|^2+|Y-Y_0|^2, 1\}  + \frac{|t-s|^2}{2\epsilon} - C|t-s|^\frac{1}{1+s} 
\end{split}
\end{equation}
Taking $\epsilon<<\alpha \min\{\sqrt{t_0}, \sqrt{T-t_0}\}$,  we can guarantee then that 
\begin{equation}
\Psi_\epsilon(t,s,X,Y) - (U_1(t,X)-U_2(s, Y)) > 0, \qquad t=0,T \text{ or } s=0,T.
\end{equation}
As the RHS of \eqref{e:viscositylemmalowerbound} $\to \infty $ as $ |(X,Y)|\to \infty$, we thus have that a global minimum $(t_\epsilon, s_\epsilon, X_\epsilon, Y_\epsilon)$ exists.  Since
\begin{equation}
\begin{split}
\Psi_\epsilon(t,s,X,Y) - (U_1(t,X)-U_2(s, Y))\geq\frac{|t-s|^2}{2\epsilon} - C|t-s| \geq  -\tilde{C}\epsilon^{\frac{1+s}{1+2s}}
\\ \Psi_\epsilon(t_0, t_0, X_0, Y_0)- (U_1(t_0,X_0)-U_2(t_0, Y_0)) = 0,
\end{split}
\end{equation} 
it then follows that at the minimum 
\begin{equation}
|X_\epsilon-X_0|^2, |Y_\epsilon-Y_0|^2 \leq \frac{\tilde{C}\epsilon^{\frac{1+s}{1+2s}}}{2\alpha}.
\end{equation}
Thus the sequence $(t_\epsilon, s_\epsilon, X_\epsilon, Y_\epsilon)$ is bounded.  Passing to a convergent subsequence, the uniqueness of the minimum for $\Psi(t,X,Y)-(U_1(t,X)-U_2(t,Y))$ then implies that 
\begin{equation}
(t_\epsilon, s_\epsilon, X_\epsilon, Y_\epsilon)\to (t_0, t_0, X_0, Y_0).
\end{equation}

Now, assume that $\nabla_X \Psi(t_0,X_0,Y_0),\nabla_Y \Psi(t_0. X_0, Y_0)\not = 0$.  Then by continuity, for $\epsilon$ sufficiently small 
\begin{equation}
\nabla_X \Psi_\epsilon(t_\epsilon, s_\epsilon, X_\epsilon, Y_\epsilon), \nabla_Y \Psi_\epsilon(t_\epsilon, s_\epsilon, X_\epsilon, Y_\epsilon) 
\not = 0.
\end{equation}

As $U_1$ is a subsolution to \eqref{e:viscositysub} and $U_2$ a viscosity supersoltuion to \eqref{e:viscositysuper}, it follows that 
\begin{equation}
\begin{split}
\partial_t \Psi_\epsilon(t_\epsilon, s_\epsilon, X_\epsilon, Y_\epsilon) + \partial_s \Psi_\epsilon(t_\epsilon, s_\epsilon,& X_\epsilon, Y_\epsilon)\leq 
\\ &-H_s(X_\epsilon, \{\Psi_\epsilon(t_\epsilon, s_\epsilon, \cdot, Y_\epsilon)\geq \Psi_\epsilon(t_\epsilon, s_\epsilon, X_\epsilon, Y_\epsilon)\})|\nabla_X\Psi_\epsilon(t_\epsilon, s_\epsilon, X_\epsilon, Y_\epsilon)|
\\&+H_s(Y_\epsilon, \{\Psi_\epsilon(t_\epsilon, s_\epsilon, X_\epsilon, \cdot)> \Psi_\epsilon(t_\epsilon, s_\epsilon, X_\epsilon, Y_\epsilon)\})|\nabla_Y\Psi_\epsilon(t_\epsilon, s_\epsilon, X_\epsilon, Y_\epsilon)|
\end{split}
\end{equation}
We have by continuity that 
\begin{equation}
\begin{split}
\partial_t \Psi_\epsilon(t_\epsilon, s_\epsilon, X_\epsilon, Y_\epsilon) + \partial_s \Psi_\epsilon(t_\epsilon, s_\epsilon, X_\epsilon, Y_\epsilon) = \frac{\partial_t \Psi(t_\epsilon, X_\epsilon, Y_\epsilon) + \partial_t\Psi( s_\epsilon,X_\epsilon, Y_\epsilon) }{2} \to \partial_t \Psi(t_0, X_0, Y_0),
\\ |\nabla_X\Psi_\epsilon(t_\epsilon, s_\epsilon, X_\epsilon, Y_\epsilon)|\to |\nabla_X\Psi(t_0, X_0, Y_0)|, 
\\ |\nabla_Y\Psi_\epsilon(t_\epsilon, s_\epsilon, X_\epsilon, Y_\epsilon)|\to |\nabla_Y\Psi(t_0, X_0, Y_0)|,  
\end{split}
\end{equation}
As the $s$-mean curvature of closed superlevel sets $H_s(X, \{F(\cdot)\geq F(X)\})$ is upper semicontinuous by Lemma \ref{l:fractmeancurvsemicont}, it then follows that 
\begin{equation}
\begin{split}
\partial_t\Psi(t_0, X_0, Y_0) = &\lim\limits_{\epsilon\to 0} \partial_t \Psi_\epsilon(t_\epsilon, s_\epsilon, X_\epsilon, Y_\epsilon) + \partial_s \Psi_\epsilon(t_\epsilon, s_\epsilon, X_\epsilon, Y_\epsilon) 
\\ \leq &\limsup\limits_{\epsilon\to 0}  -H_s(X_\epsilon, \{\Psi_\epsilon(t_\epsilon, s_\epsilon, \cdot, Y_\epsilon)\geq \Psi_\epsilon(t_\epsilon, s_\epsilon, X_\epsilon, Y_\epsilon)\})|\nabla_X\Psi_\epsilon(t_\epsilon, s_\epsilon, X_\epsilon, Y_\epsilon)|
\\&\qquad+H_s(Y_\epsilon, \{\Psi_\epsilon(t_\epsilon, s_\epsilon, X_\epsilon, \cdot)> \Psi_\epsilon(t_\epsilon, s_\epsilon, X_\epsilon, Y_\epsilon)\})|\nabla_Y\Psi_\epsilon(t_\epsilon, s_\epsilon, X_\epsilon, Y_\epsilon)|
\\ &\leq -H_s(X_0, \{\Psi(t_0 \cdot, Y_0)\geq \Psi(t, X_0, Y_0)\})|\nabla_X\Psi(t_0, X_0, Y_0)|
\\&\qquad+H_s(Y_0, \{\Psi(t_0, X_0, \cdot)> \Psi(t_0,  X_0, Y_0)\})|\nabla_Y\Psi(t_0,  X_0, Y_0)|.  
\end{split}
\end{equation}

The case that $\nabla_X \Psi(t_0,X_0,Y_0),\nabla_Y \Psi(t_0. X_0, Y_0)\not = 0$ is sufficient for our use in Lemma \ref{l:assumptionscase}.  Else, it can be proved similarly to the proof of Lemma 3 in \cite{CyrilFract}.  

\end{proof}

{\bf Proof of Lemma \ref{l:viscositydouble} }

\begin{proof}
Without loss of generality, we assume that $\gamma = 0$ and  $\mathcal{L}^d(\Gamma_t ) = 0$ for almost every $t\in \R$.  Hence, $\Gamma_t$ cannot fatten and has empty interior for all times $t$.  

Let $(E_t^-, \Gamma_t, E_t^+)$ be our viscosity solution.  For $n\in\N$, let $U^{(n\pm)}(t,X)$ be the viscosity solution of \eqref{e:viscosity} for the initial data 
\begin{equation}
\begin{split}
U_0^{(n+)}(X) &= \left\{\begin{array}{ll} 0, & X\in \overline{E_0^+}, \\ \min\{1, n d(X, \Gamma_0)\} , & X\in E_0^-\cup \Gamma_0, \end{array} \right. ,
\\ U_0^{(n-)}(Y) &= - U_0^{(n+)}(Y)
\end{split}
\end{equation}

By Lemmas \ref{l:viscositydoublesmooth} and \ref{l:compactness}, we have that $U^{(n+)}(t,X)-U^{(n-)}(t,Y)$ is a viscosity subsolution of the doubled equation \eqref{e:viscositydouble}.  We claim that
\begin{equation}\label{e:semicontenv}
\1_{E_t^-\cup\Gamma_t}(X) - \1_{E_t^{-}}(Y) = \limsup\limits_{(t_n, X_n, Y_n)\to (t,X, Y)} U^{(n+)}(t_n,X_n)-U^{(n-)}(t_n,Y_n),
\end{equation}
and hence that $\1_{E_t^-\cup\Gamma_t}(X) - \1_{E_t^{-}}(Y) $ is the upper semicontinuous envelope of subsolutions.  Thus is a subsolution itself by standard viscosity solution theory.  

As $U_0^{(n+)}(X) \leq \1_{E_0^-\cup\Gamma_0}(X)$ and $U_0^{(n-)}(Y))\geq \1_{E_0^-}(Y)$, we have immediately by Theorem \ref{t:subsuper} that 
\begin{equation}
0\leq \limsup\limits_{(t_n,X_n)\to (t,X)} U^{(n+)}(t_n,X_n) \leq \1_{E_t^-\cup\Gamma_t}(X), \qquad  \liminf\limits_{(t_n,Y_n)\to (t,Y)} U^{(n-)}(t_n,Y_n) \leq \1_{E_t^-}(Y)\leq 1,
\end{equation}

We claim that in fact
\begin{equation}\label{e:enough}
 \lim\limits_{t_n\to t}\limsup\limits_{X_k\to X} U^{(n+)}(t_n,X_k) = \1_{E_t^-\cup\Gamma_t}(X), 
\end{equation}
which implies \eqref{e:semicontenv}.  To prove this, it suffices to consider the case that $(t,X)$ is in the interior of $E_t^-\cup \Gamma_t$.  As $\Gamma_t$ does not fatten by assumption, that means that $(t,X)\in E_t^-$.  As $E_t^{-} = \displaystyle\bigcup\limits_{\gamma<0} E_t^{\gamma -}$, we thus have that $X\in E_t^{\gamma-}$ for some $\gamma<0$.  As $U^{n+}_0(Z)\geq \1_{E_0^{\gamma}\cup\Gamma_0^\gamma}(Z)$ for $Z$ sufficiently large, we then have by the comparison principle that 
\begin{equation}
1=\1_{E_t^{\gamma-}\cup\Gamma_t^\gamma}(X) \leq U^{(n+)}(t,X),
\end{equation}
for all $n$ sufficiently large.  Since $U^{(n^+)}$ is continuous, that then implies \eqref{e:enough}.

\end{proof}

\end{appendix}

\bibliography{Fractional-Mean-Curvature-Flow-References}{}
\bibliographystyle{alpha}

\end{document}